\documentclass[twoside,reqno]{amsart}
\usepackage{amssymb,enumerate}
\usepackage{amsmath,amsthm}
\usepackage{amsfonts}
\usepackage{color}
\usepackage{float}
\usepackage{graphicx}
\usepackage{xy}
\usepackage{ifpdf}
\usepackage{ccaption}

\xyoption{all}
\captionnamefont{\bfseries}
\captiondelim{. }
\hangcaption

\ifpdf  \fi

\theoremstyle{definition}
\newtheorem{definition}{Definition}[section]
\theoremstyle{remark}
\newtheorem{remark}[definition]{Remark}

\newtheoremstyle{theorem}{2mm}{2mm}{\slshape}{}{\bfseries}{}{ }{\thmname{#1} \thmnumber{#2}. \thmnote{ #3}}
\theoremstyle{theorem}
\newtheorem{theorem}[definition]{Theorem}
\newtheorem{corollary}[definition]{Corollary}
\newtheorem{lemma}[definition]{Lemma}
\newtheorem{proposition}[definition]{Proposition}

\newtheoremstyle{mystyle}{2mm}{2mm}{}{}{\bfseries}{}{ }{\thmnumber{#2}.\thmnote{ #3}}
\theoremstyle{mystyle}
\newtheorem{fact}[definition]{}

\newcommand{\sv}[1]{\mbox{\fontsize{9}{11}\selectfont $#1$}}
\newcommand{\sa}[1]{\mbox{\fontsize{7}{7}\selectfont $#1$}}

\begin{document}
\title{Factorizable enriched categories and applications}
\author{Aura B\^{a}rde\c{s}}
\author{Drago\c{s} \c{S}tefan}
\address{University of Bucharest, Faculty of Mathematics and Informatics,
Bucharest, 14 Academiei Street, Ro-010014, Romania.}
\email{aura\_bardes@yahoo.com}
\address{University of Bucharest, Faculty of Mathematics and Informatics,
Bucharest, 14 Academiei Street, Ro-010014, Romania.}
\email{drgstf@gmail.com}
\date{}

\begin{abstract}
We define the twisted tensor product of two enriched categories, which
generalizes various sorts of `products' of algebraic structures, including
the bicrossed product of groups, the twisted tensor product of (co)algebras
and the double cross product of bialgebras. The key ingredient in the
definition is the notion of simple twisting systems between two enriched
categories. To give examples of simple twisted tensor products we introduce
matched pairs of enriched categories. Several other examples related to
ordinary categories, posets and groupoids are also discussed.
\end{abstract}

\maketitle
\tableofcontents

\section*{Introduction}

The most convenient way to explain what we mean by the factorization problem
of an algebraic structure is to consider a concrete example. Chronologically
speaking, the first problem of this type was studied for groups, see for
instance \cite{Mai,Ore,Za,Sze,Tak}. Let $G$ be a group. Let $H$ and $K$
denote two subgroups of $G.$ One says that $G$ factorizes through $H$ and $K$
if $G=HK$ and $H\cap K=1.$ Therefore, the factorization problem for groups
means to find necessary and sufficient conditions which ensure that $G$
factorizes through the given subgroups $H$ and $K$. Note that, if $G$
factorizes through $H$ and $K$ then the multiplication induces a canonical
bijective map $\varphi :H\times K\rightarrow G,$ which can be used to
transport the group structure of $G$ on the Cartesian product of $H$ and $K.$
We shall call the resulting group structure the bicrossed product of $H$ and
$K$, and we shall denote it by $H\Join K$. The identity element of $H\Join K$
is $(1,1)$, and its group law is uniquely determined by the `twisting' map
\begin{equation*}
R:K\times H\rightarrow H\times K,\quad R(k,h):=\varphi ^{-1}(kh).
\end{equation*}%
Obviously, $R$ is induced by a couple of functions $\triangleright :K\times
H\rightarrow H$ and $\triangleleft :K\times H\rightarrow K$ such that $%
R(k,h)=(k\triangleright h,k\triangleleft h).$ Using this notation the
multiplication on $H\Join K$ can be written as
\begin{equation*}
(h,k)\cdot (h^{\prime },k^{\prime })=\left( h(k\triangleright h^{\prime
}),(k\triangleleft h^{\prime })k^{\prime }\right) .
\end{equation*}%
The group axioms easily imply that $(H,K,\triangleright ,\triangleleft )$ is
a matched pair of groups, in the sense of \cite{Tak}. Conversely, any
bicrossed product $H\Join K$ factorizes through $H$ and $K.$ In conclusion,
a group $G$ factorizes through $H$ and $K$ if and only if it is isomorphic
to the bicrossed product $H\Join K$ associated to a certain matched pair $%
(H,K,\triangleright ,\triangleleft ).$

Similar `products' are known in the literature for many other algebraic
structures. In \cite{Be}, for a distributive law $\lambda:G\circ
F\rightarrow F\circ G$ between two monads, Jon Beck defined a monad
structure on $F\circ G,$ which can be regarded as a sort of bicrossed
product of $F$ and $G$ with respect to the twisting natural transformation $%
\lambda.$

The twisted tensor product of two $\mathbb{K}$-algebras $A$ and $B$ with
respect to a $\mathbb{K}$-linear twisting map $R:B\otimes _{\mathbb{K}%
}A\rightarrow A\otimes _{\mathbb{K}}B$ was investigated for instance in \cite%
{Ma1}, \cite{Tam}, \cite{CSV}, \cite{CIMZ}, \cite{LPvO} and \cite{JLPvO}. It
is the analogous in the category of associative and unital algebras of the
bicrossed product of groups. The classical tensor product of two algebras,
the graded tensor product of two graded algebras, skew algebras, smash
products, Ore extensions, generalized quaternion algebras, quantum affine
spaces and quantum tori are all examples of twisted tensor products.

Another class of examples, including the Drinfeld double and the double
crossed product of a matched pair of bialgebras, comes from the theory of
Hopf algebras, see \cite{Ma2}. Some of these constructions have been
generalized for bialgebras in monoidal categories \cite{BD} and bimonads
\cite{BV}.

Enriched categories have been playing an increasingly important role not
only in Algebra, but also in Algebraic Topology and Mathematical Physics,
for instance. They generalize usual categories, linear categories, Hopf
module categories and Hopf comodule categories. Monoids, algebras,
coalgebras and bialgebras may be regarded as enriched categories with one
object.

Our aim in this paper is to `categorify' the factorization problem, i.e. to
answer the question when an enriched category factorizes through a couple of
enriched subcategories. Finding a solution at this level of generality would
allow us to approach in an unifying way all factorization problems that we
have already mentioned. Moreover, it would also provide a general method for
producing new non-trivial examples of enriched categories.

In order to define factorizable enriched categories, we need some notation.
Let $\boldsymbol{C}$ be a small enriched category over a monoidal category $(%
\boldsymbol{M},\otimes ,\mathbf{1}).$ Let $S$ denote the set of objects in $%
\boldsymbol{C}$. For the hom-objects in $\boldsymbol{C}$ we use the notation
$_{x}C_{y}.$ The composition of morphisms and the identity morphisms in $%
\boldsymbol{C}$ are defined by the maps $_{x}c_{z}^{y}:{}_{x}C_{y}\otimes
{}_{y}C_{z}\rightarrow {}_{x}C_{z}$ and $1_{x}:\mathbf{1}\rightarrow
{}_{x}C_{x},$ respectively. For details, the reader is referred to the next
section. We assume that $\boldsymbol{A}$ and $\boldsymbol{B}$ are two
enriched subcategories of $\boldsymbol{C}$. The inclusion functor $%
\boldsymbol{\alpha} :\boldsymbol{A}\rightarrow \boldsymbol{C}$ is given by a
family $\{_{x}\alpha _{y}\}_{x,y\in S}$ of morphisms in $\boldsymbol{M},$
where $_{x}\alpha _{y}:{}_{x}{A}_{y}\rightarrow {}${}$_{x}C_{y}.$ If $%
\boldsymbol{\beta }$ is the inclusion of $\boldsymbol{B}$ in $\boldsymbol{C}$%
, then for $x,$ $y$ and $u$ in $S$ we define%
\begin{equation*}
_{x}\varphi _{y}^{u}:{}_{x}A_{u}\otimes {}_{u}B{_{y}\rightarrow {}_{x}{C}%
_{y},\quad }_{x}\varphi _{z}^{y}={}_{x}c_{y}^{u}\circ (_{x}{\alpha }%
_{u}\otimes {}_{u}{\beta }_{y}).
\end{equation*}%
{Assuming that all }$S$-indexed families of objects in $\boldsymbol{M}$ have
a coproduct it follows that the maps $\{_{x}\varphi _{y}^{u}\}_{u\in S}$
yield a unique morphism
\begin{equation*}
_{x}\varphi _{y}:{}\textstyle\bigoplus\limits_{u\in S}{}_{x}A_{u}\otimes
{}_{u}B{_{y}\rightarrow {}_{x}{C}_{y}.}
\end{equation*}%
We say that $\boldsymbol{C}$ {factorizes} {through} $\boldsymbol{A} $ {and} $%
\boldsymbol{B}$ if and only if all $_{x}\varphi _{y}$ are invertible. An
enriched category $\boldsymbol{C}$ is called {factorizable} if it factorizes
through $\boldsymbol{A}$ and $\boldsymbol{B},$ for some $\boldsymbol{A}$ and
$\boldsymbol{B}$.

In Theorem \ref{teo1}, our first main result, under the additional
assumption that the tensor product on $\boldsymbol{M}$ is distributive over
the direct sum, we show that to every $\boldsymbol{M}$-category $\boldsymbol{%
C}$ that factorizes through $\boldsymbol{A} $ and $\boldsymbol{B}$
corresponds a twisting system between $\boldsymbol{B}$ and $\boldsymbol{A},$
that is a family ${R}:=\{_{x}R_{z}^{y}\}_{x,y,z\in S}$ of morphisms
\begin{equation*}
_{x}R_{z}^{y}:{}_{x}B_{y}\otimes {}_{y}A{_{z}\rightarrow } \textstyle%
\bigoplus\limits_{u\in S}{}_{x}A_{u}\otimes {}_{u}B{_{z}}
\end{equation*}%
which are compatible with the composition and identity maps in $\boldsymbol{A%
} $ and $\boldsymbol{B}$ in a certain sense.

Trying to associate to a twisting system ${R}:=\{_{x}R_{z}^{y}\}_{x,y,z\in
S} $ an $\boldsymbol{M}$-category we encountered some difficulties due to
the fact that, in general, the image of $_{x}R_{z}^{y}$ is too big.
Consequently, in this paper we focus on the particular class of twisting
systems for which there is a function $\left\vert \cdots \right\vert
:S\times S\times S\rightarrow S$ such that the image of $_{x}R_{z}^{y}$ is
included into $_{x}A_{\left\vert xyz\right\vert }\otimes {}_{|xyz|}B{_{z},}$
for every $x,y,z\in S$. These twisting systems are characterized in
Proposition \ref{thm R simplu}. A more precise description of them is given
in Corollary \ref{cor: R simplu1}, provided that $\boldsymbol{M}$ satisfies
an additional condition (\dag ), see \S \ref{fa:dag}. A similar result is
obtained in Corollary \ref{cor R simplu2} for a linear monoidal category.

In this way we are led in \S \ref{STS} to the definition of simple twisting
systems. For such a twisting system ${R}$ between $\boldsymbol{B}$ and $%
\boldsymbol{A},$ in Theorem \ref{thm:TTP} we construct an $\boldsymbol{M}$%
-category $\boldsymbol{A}\otimes _{{R}}\boldsymbol{B}$ which factorizes
through $\boldsymbol{A}$ and $\boldsymbol{B}.$ Since it generalizes the
twisted tensor product of algebras, $\boldsymbol{A}\otimes _{{R}}\boldsymbol{%
B}$ will be called the twisted tensor product of $\boldsymbol{A}$ and $%
\boldsymbol{B}.$

In the third section we consider the case when $\boldsymbol{M}$ is the
monoidal category of coalgebras in a braided category $\boldsymbol{M}%
^{\prime }$. In this setting, we prove that there is an one-to-one
correspondence between simple twisting systems and matched pair of enriched
categories, see \S \ref{fa:MP} for the definition of the latter notion. We
shall refer to the twisted tensor product of a matched pair as the bicrossed
product. By construction, the bicrossed product is a category enriched over $%
\boldsymbol{M^{\prime }}$, but we prove that it is enriched over $%
\boldsymbol{M}$ as well.

More examples of twisted tensor products of enriched categories are given in
the last part of the paper. By definition, usual categories are enriched
over $\boldsymbol{Set},$ the category of sets. Actually, they are enriched
over the monoidal category of coalgebras in $\boldsymbol{Set}$. Hence,
simple twisting systems and matched pairs are equivalent notions for usual
categories. Moreover, if $\boldsymbol{A}$ and $\boldsymbol{B}$ are thin
categories (that is their hom-sets contain at most one morphism), then we
show that any twisting system between $\boldsymbol{B}$ and $\boldsymbol{A}$
is simple, so it corresponds to a uniquely determined matched pair of
categories. We use this result to investigate the twisting systems between
two posets.

Our results may be applied to algebras in a monoidal category $\boldsymbol{M}
$, which are enriched categories with one object. Therefore, we are also
able to recover all bicrossed product constructions that we discussed at the
beginning of this introduction.

Finally, we prove that the bicrossed product of two groupoids is also a
groupoid, and we give an example of factorizable groupoid with two objects.

\section{Preliminaries and notation.}

Mainly for fixing the notation and the terminology, in this section we
recall the definition of enriched categories, and then we give some example
that are useful for our work.

\begin{fact}[Monoidal categories.]
Throughout this paper $(\boldsymbol{M},\otimes ,\mathbf{1},a,l,r)$ will
denote a monoidal category with associativity constraints $a_{X,Y,Z}:\left(
X\otimes Y\right) \otimes Z\rightarrow X\otimes \left( Y\otimes Z\right) $
and unit constraints $l_{X}:\mathbf{1}\otimes X\rightarrow X$ and $%
r_{X}:X\otimes \mathbf{1}\rightarrow X.$ The class of objects of $%
\boldsymbol{M}$ will be denoted by $M_{0}$. Mac Lane's Coherence Theorem
states that given two parenthesized tensor products of some objects $%
X_{1},\dots ,X_{n}$ in $\boldsymbol{M}$ (with possible arbitrary insertions
of the unit object $\boldsymbol{1}$) there is a unique morphism between them
that can be written as a composition of associativity and unit constraints,
and their inverses. Consequently, all these parenthesized tensor products
can be identified coherently, and the parenthesis, associativity constraints
and unit constraints may be omitted in computations. Henceforth, we shall
always ignore them. The identity morphism of an object $X$ in $\boldsymbol{M}
$ will be denoted by the same symbol $X.$

By definition, the tensor product is a functor. In particular, for any
morphisms $f^{\prime }:X^{\prime }\rightarrow Y^{\prime \prime }$ and $%
f^{\prime \prime }:X^{\prime \prime }\rightarrow Y^{\prime \prime }$ in $%
\boldsymbol{M}$ the following equations hold%
\begin{equation}  \label{ec:TP=functor}
\left( f^{\prime }\otimes Y^{\prime \prime }\right) \circ \left( X^{\prime
}\otimes f^{\prime \prime }\right) =f^{\prime }\otimes f^{\prime \prime
}=\left( Y^{\prime }\otimes f^{\prime \prime }\right) \circ \left( f^{\prime
}\otimes X^{\prime \prime }\right) .
\end{equation}

If the coproduct of a family $\{X_{i}\}_{i\in I}$ of objects in $\boldsymbol{%
M}$ exists, then it will be denoted as a pair $\left( \textstyle%
\bigoplus\limits {}_{i\in I}X_{i},\left\{ \sigma _{i}\right\} _{i\in
I}\right) ,$ where the maps $\sigma _{i}:X_{i}\rightarrow \textstyle%
\bigoplus\limits {}_{i\in I}X_{i}$ are the canonical inclusions.
\end{fact}

\begin{fact}[The opposite monoidal category.]
\label{CatMonOpusa} If $(\boldsymbol{M},\otimes ,\mathbf{1,}a,l,r)$ is a
monoidal category, then one constructs the monoidal category $(\boldsymbol{M}%
^{o},\otimes ^{o},\mathbf{1}^{o},a^{o},l^{o},r^{o})$ as follows. By
definition, $\boldsymbol{M}^{o}$ and $\boldsymbol{M}$ share the same objects
and identity morphisms. On the other hand, for two objects $X,Y$ in $%
\boldsymbol{M}$, one takes $\mathrm{Hom}_{\boldsymbol{M}^{o}}(X,Y):={}%
\mathrm{Hom}_{\boldsymbol{M}}(Y,X)$. The composition of morphisms in $%
\boldsymbol{M}^{o}$
\begin{equation*}
\bullet :\mathrm{Hom}_{\boldsymbol{M}^{o}}(Y,Z)\times \mathrm{Hom}_{%
\boldsymbol{M}^{o}}(X,Y){}\rightarrow \mathrm{Hom}_{\boldsymbol{M}%
^{o}}(X,Z){}
\end{equation*}%
is defined by the formula $f\bullet g:=g\circ f,$ for any $f:Z\rightarrow Y$
and $g:Y\rightarrow X$ in $\boldsymbol{M}.$ The monoidal structure is
defined by $X\otimes ^{o}Y=X\otimes Y$ and $\boldsymbol{1}^{o}=\boldsymbol{1}
$. The associativity and unit constraints in $\boldsymbol{M}^{o}$ are given
by $a^{o}_{X,Y,Z}=a_{X,Y,Z}^{-1},$ $l^{o}=l_{X}^{-1}$ and $r^{o}=r_{X}^{-1}.$
If, in addition $\boldsymbol{M}$ is braided monoidal, with braiding $\chi
_{X,Y}:X\otimes Y\rightarrow Y\otimes X$ then $\boldsymbol{M}^{o}$ is also
braided, with respect to the braiding $\chi ^{o}$ defined by $\chi
^{o}_{X,Y}:=( \chi _{X,Y})^{-1}.$
\end{fact}

\begin{definition}
Let $S$ be a set. We say that a monoidal category $\boldsymbol{M}$ is $S$%
\emph{-distributive} if every $S$-indexed family of objects in $\boldsymbol{%
M\ }$has a coproduct, and the tensor product is distributive to the left and
to the right over any such coproduct. More precisely, $\boldsymbol{M}$ is $S$%
-distributive if for any family $\left\{ X_{i}\right\} _{i\in S}$ the
coproduct $(\textstyle\bigoplus\nolimits_{i\in S}X_{i},\{\sigma _{i}\}_{i\in
S})$ exists and, for an arbitrary object $X$,
\begin{equation*}
(X\otimes (\textstyle\bigoplus\limits_{i\in S}X_{i}),\{X\otimes \sigma
_{i}\}_{i\in S})\text{\qquad and\qquad }((\textstyle\bigoplus\limits_{i\in
S}X_{i})\otimes X,\{\sigma _{i}\otimes X\}_{i\in S})
\end{equation*}%
are the coproducts of $\{X\otimes X_{i}\}_{i\in S}$ and $\{X_{i}\otimes
X\}_{i\in S},$ respectively. Note that all monoidal categories are $S$%
-distributive, provided that $S$ is a singleton (i.e. the cardinal of $S$ is
1).
\end{definition}

\begin{fact}[Enriched categories.]
An \emph{enriched category} $\boldsymbol{C}$ over $(\boldsymbol{M},\otimes ,%
\mathbf{1}),$ or an $\boldsymbol{M}$-category for short, consists of:

\begin{enumerate}
\item A class of objects, that we denote by ${C}_{0}.$ If ${C}_{0}$ is a set
we say that $\boldsymbol{C}$ is \emph{small}.

\item A hom-object $_{x}{C}_{y}$ in $\boldsymbol{M}$, for each $x$ and $y$
in ${C}_{0}.$ It plays the same role as $\mathrm{Hom}_{\boldsymbol{C}}(y,x)$%
, the set of morphisms from $y$ to $x$ in an ordinary category $\boldsymbol{C%
}$.

\item A morphism $_{x}c_{z}^{y}:{}_{x}{C}_{y}\otimes {}_{y}{C}%
_{z}\rightarrow {}_{x}{C}_{z},$ for all $x,$ $y$ and $z$ in ${C}_{0}.$

\item A morphism $1_{x}:\mathbf{1\rightarrow {}}_{x}{C}_{x},$ for all $x$ in
${C}_{0}.$
\end{enumerate}

\noindent By definition one assumes that the diagrams in Figure~\ref{ec:
M-categorie} are commutative, for all $x,y,z$ and $t$ in ${C}_{0}.$ The
commutativity of the square means that the composition of morphisms in $%
\boldsymbol{C}$, defined by $\left\{ _{x}c_{z}^{y}\right\} _{z,y,z\in C_{0}}$%
, is \emph{associative}. We shall say that $1_{x}$ is the \emph{identity
morphism} of $x\in C_{0}$.

\begin{figure}[h]
\begin{equation*}
\xymatrix{ {}\sv{_x{C}_y\otimes{}_y{C}_z\otimes{}_z{C}_t}
\ar[rr]^-{_x{c}^y_z\otimes{}_z{C}_t}\ar[d]_-{_x{C}_y\otimes{}_y{c}^z_t} &
&\sv{_x{C}_z\otimes{}_z{C}_t}\ar[d]^-{_x{c}^z_t}\\
\sv{_x{C}_y\otimes{}_y{C}_t}\ar[rr]_-{_x{c}^y_t} & &\sv{_x{C}_t}}\qquad %
\xymatrix{ \sv{_x{C}_y\otimes{}_y{C}_y} \ar[rd]_-{_xc^y_y} &\sv{_x{C}_y}
\ar[l]_-{_xC_y\otimes{}1_y}\ar[r]^-{1_x\otimes{}_xC_y}
&\sv{_x{C}_x\otimes{}_x{C}_y}\ar[ld]^-{_xc^x_y}\\ & \sv{_x{C}_y}\ar@{=}[u] &
}
\end{equation*}%
\caption{The definition of enriched categories.}
\label{ec: M-categorie}
\end{figure}

An $\boldsymbol{M}$-functor $\boldsymbol{\alpha }:\boldsymbol{C}\rightarrow
\boldsymbol{C}^{\prime }$ is a couple $(\alpha _{0},\{_{x}\alpha
_{y}\}_{x,y\in C_{0}})$, where $\alpha _{0}:C_{0}\rightarrow C_{0}^{\prime }$
is a function and $_{x}\alpha _{y}:{}_{x}C_{y}\rightarrow {}_{x^{\prime
}}C^{\prime }_{y^{\prime }}$ is a morphism in $\boldsymbol{M}$ for any $%
x,y\in C_{0},$ where for simplicity we denoted $\alpha _{0}(u)$ by $%
u^{\prime }$, for any $u\in C_{0}.$ By definition, $\alpha _{0}$ and $%
_{x}\alpha _{y}$ must satisfy the following conditions
\begin{equation*}
_{x}\alpha _{x}\circ 1_{x}^{C}=1_{x^{\prime }}^{D}\qquad \text{and\qquad }%
_{x^{\prime }}d_{z^{\prime }}^{y^{\prime }}\circ (_{x}\alpha _{y}\otimes
{}_{y}\alpha _{z})={}_{x}\alpha _{z}\circ {}_{x}c_{z}^{y}.
\end{equation*}
\end{fact}

\begin{fact}
\label{fa:not} To work easier with tensor products of hom-objects in $%
\boldsymbol{M}$-categories we introduce some new notation. Let $S$ be a set
and for every $i=1,\dots ,n+1$ we pick up a family $\left\{
_{x}X_{y}^{i}\right\} _{x,y\in S}$ of objects in $\boldsymbol{M}$. If $%
x_{1},\dots ,x_{n+1}\in S$ then the tensor product $_{x_{0}}X_{x_{1}}^{1}%
\otimes {}_{x_{1}}X_{x_{2}}^{2}\otimes \cdots \otimes
{}_{x_{n-1}}X_{x_{n}}^{n}\otimes {}_{x_{n}}X_{x_{n+1}}^{n+1}$ will be
denoted by $_{x_{0}}X_{x_{1}}^{1}X_{x_{2}}^{2}\cdots
{}{}_{x_{n}}X_{x_{n+1}}^{n+1}.$ Assuming that $\boldsymbol{M}$ is $S$%
-distributive and fixing $x_{0}$ and $x_{n+1},$ one can construct
inductively the iterated coproduct
\begin{equation}
_{x_{0}}X_{\overline{x}_{1}}^{{1}}\cdots {}_{\overline{x}_{n-1}}X_{\overline{%
x}_{n}}^{n}X_{x_{n+1}}^{n+1}:=\textstyle\bigoplus_{x_{1}\in S}\cdots %
\textstyle\bigoplus_{x_{n}\in S}{}_{x_{0}}X_{x_{1}}^{1}\cdots
{}_{x_{n-1}}X_{x_{n}}^{n}X_{x_{n+1}}^{n+1}.  \label{cop}
\end{equation}%
It is not difficult to see that this object is a coproduct of $%
\{_{x_{0}}X_{x_{1}}^{1}\cdots {}_{x_{n}}X_{x_{n+1}}^{n+1}\}_{\left(
x_{1},\dots ,x_{n}\right) \in S^{n}}.$ Moreover, as a consequence of the
fact that the tensor product is distributive over the direct sum, we have
\begin{equation}
_{x_{0}}X_{\overline{x}_{1}}^{1}\cdots {}_{\overline{x}_{n-1}}X_{\overline{x}%
_{n}}^{n}X_{x_{n+1}}^{n+1}\cong \textstyle\bigoplus_{x_{1}\in S}\cdots %
\textstyle\bigoplus_{x_{n}\in S}{}_{x_{0}}X_{x_{\pi (1)}}^{1}X_{x_{\pi
(2)}}^{2}\cdots {}_{x_{\pi (n)}}X_{x_{n+1}}^{n+1}  \label{ec:suma-tensor}
\end{equation}%
for any permutation $\pi $ of the set $\{1,2,\dots ,n\}.$ The inclusion of $%
_{x_{0}}X_{x_{1}}^{1}\cdots {}_{x_{n}}X_{x_{n+1}}^{n+1}$ into the coproduct
defined in \eqref{cop} is also inductively constructed as the composition of
the following two arrows%
\begin{equation*}
_{x_{0}}X_{x_{1}}^{1}\otimes {}_{x_{1}}X_{x_{2}}^{1}\cdots
{}_{x_{n}}X_{x_{n+1}}^{n+1}\longrightarrow {}_{x{}_{0}}X_{x_{1}}^{1}\otimes
{}_{x_{1}}X_{\overline{x}_{2}}^{1}\cdots {}_{\overline{x}%
_{n}}X_{x_{n+1}}^{n+1}\hookrightarrow \textstyle\bigoplus_{x_{1}\in
S}{}_{x{}_{0}}X_{x_{1}}^{1}\otimes {}_{x_{1}}X_{\overline{x}_{2}}^{1}\cdots
{}_{\overline{x}_{n}}X_{x_{n+1}}^{n+1}{},
\end{equation*}%
where the first morphism is the tensor product between the identity of $%
_{x_{0}}X_{x_{1}}^{1}$ and the inclusion of $_{x_{1}}X_{x_{2}}^{1}\cdots
{}_{x_{n}}X_{x_{n+1}}^{n+1}$ into $_{x_{1}}X_{\overline{x}_{2}}^{1}\cdots
{}_{\overline{x}_{n}}X_{x_{n+1}}^{n+1}.$ Clearly, for every $x_{n+1}\in S,$%
\begin{equation*}
_{\overline{x}_{0}}X_{\overline{x}_{1}}^{1}\cdots {}_{\overline{x}_{n-1}}X_{%
\overline{x}_{n}}^{n}X_{x_{n+1}}^{n+1}:=\textstyle\bigoplus_{x_{0}\in
S}{}_{x_{0}}X_{\overline{x}_{1}}^{1}\cdots {}_{\overline{x}_{n-1}}X_{%
\overline{x}_{n}}^{n}X_{x_{n+1}}^{n+1}
\end{equation*}%
is the coproduct of $\{_{x_{0}}X_{x_{1}}^{1}\cdots
{}_{x_{n}}X_{x_{n+1}}^{n+1}\}_{\left( x_{0},x_{1},\dots ,x_{n}\right) \in
S^{n}}.$ The objects $_{x_{0}}X_{\overline{x}_{1}}^{1}\cdots {}_{\overline{x}%
_{n-1}}X_{\overline{x}_{n}}^{n}X_{\overline{x}_{n+1}}^{n+1}$ and $_{%
\overline{x}_{0}}X_{\overline{x}_{1}}^{1}\cdots {}_{\overline{x}_{n-1}}X_{%
\overline{x}_{n}}^{n}X_{\overline{x}_{n+1}}^{n+1}$ are analogously defined.

A similar notation will be used for morphisms. Let us suppose that $%
_{x}\alpha _{y}^{i}{}{}$ is a morphism in $\boldsymbol{M}$ with source $%
{}_{x}X_{y}^{i}$ and target $_{x}Y_{y}^{i},$ where $x,y\in S$ and $i\in
\{1,\dots ,n+1\}.$ We set
\begin{equation*}
_{x_{1}}\alpha _{x_{2}}^{1}\alpha _{x_{3}}^{2}\cdots {}_{x_{n}}\alpha
_{x_{n+1}}^{n+1}:={}_{x_{0}}\alpha _{x_{1}}^{1}\otimes \dots \otimes
{}_{x_{n}}\alpha _{x_{n+1}}^{n+1}.
\end{equation*}%
By the universal property of coproducts, $\{_{x_{0}}\alpha
_{x_{1}}^{1}\cdots {}_{x_{n}}\alpha _{x_{n+1}}^{n+1}\}_{\left( x_{1}\cdots
x_{n}\right) \in S^{n}}$ induces a unique map $_{x_{0}}\alpha _{\overline{x}%
_{1}}^{1}\cdots {}_{\overline{x}_{n-1}}\alpha _{\overline{x}%
_{n}}^{n-1}\alpha _{x_{n+1}}^{n}\ $that commutes with the inclusions. In a
similar way one constructs
\begin{equation*}
_{\overline{x}_{0}}\alpha _{\overline{x}_{1}}^{1}\cdots {}_{\overline{x}%
_{n-1}}\alpha _{\overline{x}_{n}}^{n}\alpha _{x_{n+1}}^{n+1},\qquad
{}_{x_{0}}\alpha _{\overline{x}_{1}}^{1}\cdots {}_{\overline{x}_{n-1}}\alpha
_{\overline{x}_{n}}^{n}\alpha _{\overline{x}_{n+1}}^{n+1}\qquad \text{and}%
\qquad {}_{\overline{x}_{0}}\alpha _{\overline{x}_{1}}^{1}\cdots {}_{%
\overline{x}_{n-1}}\alpha _{\overline{x}_{n}}^{n}\alpha _{\overline{x}%
_{n+1}}^{n+1}.
\end{equation*}%
To make the above notation clearer, let us have a look at some examples. Let
$\boldsymbol{A}$ and $\boldsymbol{B}$ be two $\boldsymbol{M}$-categories
such that $A_{0}=B_{0}=S$. Recall that the hom-objects in $\boldsymbol{A}$
and $\boldsymbol{B}$ are denoted by $_{x}A_{y}$ and $_{x}B_{y}$. Hence, $_{%
\overline{x}}{A}_{y}=\textstyle\bigoplus\limits_{x\in S}{}{}{_{x}A_{y}.}$ We
also have ${_{x}A_{y}B_{z}A_{t}}={}_{x}{A}_{y}\otimes {}_{y}{B}_{z}\otimes
{}_{z}{A}_{t}$ and%
\begin{equation*}
{_{x}A_{\overline{y}}B_{\overline{z}}A_{t}=}\textstyle\bigoplus\limits_{y\in
S}\textstyle\bigoplus\limits_{z\in S}{_{x}A_{y}B_{z}A_{t}\cong }\textstyle%
\bigoplus\limits_{z\in S}\textstyle\bigoplus\limits_{y\in S}{%
_{x}A_{y}B_{z}A_{t}\cong }\textstyle\bigoplus\limits_{y,z\in S}{%
_{x}A_{y}B_{z}A_{t}.}
\end{equation*}%
\noindent Since we have agreed to use the same notation for an object and
its identity map, we can write $_{x}B_{y}\alpha _{z}A_{t}\beta _{u}$ instead
of $Id_{_{x}B_{y}}\otimes {}_{y}\alpha _{z}\otimes Id_{{}_{z}A_{t}}\otimes
{}_{t}\beta _{u}$, for any morphisms $_{y}\alpha _{z}$ and $_{t}\beta _{u}$
in $\boldsymbol{M}.$ The maps $_{\overline{x}}a_{z}^{y}:{}_{\overline{x}%
}A_{y}A_{z}\longrightarrow {}_{\overline{x}}A_{z}$ and $_{x}a_{\overline{z}%
}^{y}:{}_{x}A_{y}A_{\overline{z}}\longrightarrow {}_{x}A_{\overline{z}}$ are
induced by the composition in $\boldsymbol{A}$, that is by the set $%
\{_{x}a_{z}^{y}\}_{z\in S}$. For example, the former map is uniquely defined
such that its restriction to $_{x}A_{y}A_{z}$ and $\sigma _{x,z}\circ
\,_{x}a_{z}^{y}$ coincide for all $x\in S,$ where $\sigma _{x,z}$ is the
inclusion of $_{x}A_{z}$ into $_{\overline{x}}A_{z}.$ Similarly, $_{x}a_{z}^{%
\overline{y}}:$ ${}_{x}A_{\overline{y}}A_{z}\longrightarrow {}_{x}A_{z}$ is
the unique map whose restriction to $_{x}A_{y}A_{z}$ is $_{x}a_{z}^{y},$ for
all $y\in S.$

For more details on enriched categories the reader is referred to \cite{Ke}.
We end this section giving some examples of enriched categories.
\end{fact}

\begin{fact}[The category $\boldsymbol{Set}$.]
The category of sets is monoidal with respect to the Cartesian product. The
unit object is a fixed singleton set, say $\{\emptyset \}.$ The coproduct in
$\boldsymbol{Set}$ is the disjoint union. Since the disjoint union and the
Cartesian product commute, $\boldsymbol{Set}$ is $S$-distributive for any
set $S$. Clearly, a $\boldsymbol{Set}$-category is an ordinary category. If $%
\boldsymbol{C}$ is such a category, then an element $f\in {}_{x}{C}_{y}$
will be thought of as a morphism from $y$ to $x,$ and it will be denoted by $%
f:y\rightarrow x,$ as usual. In this case we shall say that $y$
(respectively $x)$ is the domain or the source (respectively the codomain or
the target) of $f.$ The same notation and terminology will be used for
arbitrary $\boldsymbol{M}$-categories, whose objects are sets.
\end{fact}

\begin{fact}[The category $\mathbb{K}$-$\boldsymbol{Mod}$.]
Let $\mathbb{K}$ be a commutative ring. The category of $\mathbb{K}$-modules
is monoidal with respect to the tensor product of $\mathbb{K}$-modules. The
unit object is $\mathbb{K},$ regarded as a $\mathbb{K}$-module. This
monoidal category is $S$-distributive for any $S$. By definition, a $\mathbb{%
K}$-\emph{linear category} is an enriched category over $\mathbb{K}$-$%
\boldsymbol{Mod}$.
\end{fact}

\begin{fact}[The category $\Lambda $-$\boldsymbol{Mod}$-$\Lambda .$]
Let $\Lambda $ be a $\mathbb{K}$-algebra and let $\Lambda $-$\boldsymbol{Mod}
$-$\Lambda $ denote the category of left (or right) modules over $\Lambda
\otimes _{\mathbb{K}}\Lambda ^{o},$ where $\Lambda ^{o}$ is the opposite
algebra of $\Lambda .$ Thus, $M$ is an object in $\Lambda $-$\boldsymbol{Mod}
$-$\Lambda $ if, and only if, it is a left and a right $\Lambda $-module and
these structures are compatible in the sense that
\begin{equation*}
a\cdot m=m\cdot a\text{\qquad and\qquad }(x\cdot m)\cdot y=x\cdot (m\cdot y)
\end{equation*}%
for all $a\in \mathbb{K},$ $x,y\in \Lambda $ and $m\in M.$ A morphism in $%
\Lambda $-$\boldsymbol{Mod}$-$\Lambda $ is a map of left and right $\Lambda $%
-modules. The category of $\Lambda $-bimodules is monoidal with respect to $%
(-)\otimes _{\Lambda }(-).$ The unit object in $\Lambda $-$\boldsymbol{Mod}$-%
$\Lambda $ is $\Lambda ,$ regarded as a $\Lambda $-bimodule. This monoidal
category also is $S$-distributive for any $S$.
\end{fact}

\begin{fact}[The category $H$-$\boldsymbol{Mod}.$]
Let $H$ be a bialgebra over a commutative ring $\mathbb{K}$. The category of
left $H$-modules is monoidal with respect to $(-)\otimes _{\mathbb{K}}(-)$.
If $M$ and $N$ are $H$-modules, then the $H$-action on $M\otimes N$ is given
by
\begin{equation*}
h\cdot m\otimes n=\textstyle\sum {}h_{\left( 1\right) }\cdot m\otimes
h_{\left( 2\right) }\cdot n.
\end{equation*}%
In the above equation we used the $\Sigma $-notation $\Delta \left( h\right)
=\textstyle\sum {}$ $h_{\left( 1\right) }\otimes h_{\left( 2\right) }$. The
unit object is $\mathbb{K}$, which is an $H$-module with the trivial action,
induced by the counit of $H$. This category is $S$-distributive, for any $S$%
. An enriched category over $H$-$\boldsymbol{Mod}$ is called $H$-\emph{%
module category}.
\end{fact}

\begin{fact}[The category $\boldsymbol{Comod}$-$H$.]
Dually, the category of right $H$-comodules is monoidal with respect to $%
(-)\otimes _{\mathbb{K}}(-)$. The coaction on and $M\otimes _{\mathbb{K}}N$
is defined by
\begin{equation*}
\rho (m\otimes n)=\textstyle\sum {}m_{\langle 0\rangle }\otimes n_{\langle
0\rangle }\otimes m_{\langle 1\rangle }n_{\langle 1\rangle },
\end{equation*}%
where $\rho (m)=\textstyle\sum {}m_{\langle 0\rangle }\otimes n_{\langle
0\rangle }$, and a similar $\Sigma $-notation was used for $\rho (n).$ This
category is $S$-distributive, for any set $S$. By definition, an $H$-\emph{%
comodule category} is an enriched category over $\boldsymbol{Comod}$-$H$.
\end{fact}

\begin{fact}[{The category $[\boldsymbol{A},\boldsymbol{A}].$}]
\label{Cat[A,A]}Let $\boldsymbol{A}$ be a small category, and let $[%
\boldsymbol{A},\boldsymbol{A}]$ denote the category of all endofunctors of $%
\boldsymbol{A}$. Therefore, the objects in $[\boldsymbol{A},\boldsymbol{A}]$
are functors $F:\boldsymbol{A}\rightarrow \boldsymbol{A},$ while the set $%
_{F}[\boldsymbol{A},\boldsymbol{A}]_{G}$ contains all natural
transformations $\mu :G\rightarrow F$. The composition in this category is
the composition of natural transformations. The category $[\boldsymbol{A},%
\boldsymbol{A}]$ is monoidal with respect to the composition of functors. If
$\mu :F\rightarrow G$ and $\mu ^{\prime }:F^{\prime }\rightarrow G^{\prime }$
are natural transformations, then the natural transformations $\mu F^{\prime
}$ and $G\mu ^{\prime }$ are given by
\begin{align*}
\mu F^{\prime }:F\circ F^{\prime }\rightarrow G\circ G^{\prime }& ,\qquad
\left( \mu F^{\prime }\right) _{x}:=\mu _{F^{\prime }(x)}, \\
G\mu ^{\prime }:G\circ F^{\prime }\rightarrow G\circ G^{\prime }& ,\qquad
\left( G\mu ^{\prime }\right) _{x}:=G(\mu _{x}^{\prime }).
\end{align*}%
We can now define the tensor product of $\mu $ and $\mu ^{\prime }$ by
\begin{equation*}
\mu \otimes \mu ^{\prime }:=G\mu ^{\prime }\circ \mu F^{\prime }=\mu
G^{\prime }\circ F\mu ^{\prime }.
\end{equation*}%
Even if $\boldsymbol{A}$ is $S$-distributive, $[\boldsymbol{A},\boldsymbol{A}%
]$ may not have this property. In spite of the fact that, by assumption, any
$S$-indexed family in $[\boldsymbol{A},\boldsymbol{A}]$ has a coproduct, in
general this does not commute with the composition of functors.
Nevertheless, as we have already noticed, $[\boldsymbol{A},\boldsymbol{A}]$
is $S$-distributive if $\left\vert S\right\vert =1$.

This remark will allow us to apply our main results to an $[\boldsymbol{A},%
\boldsymbol{A}]$-category $\boldsymbol{C}$ with one object $x$. Hence $%
F:={}_{x}{C}_{x}$ is an endofunctor of $\boldsymbol{A}$, and the composition
and the identity morphisms in $\boldsymbol{C}$ are uniquely defined by
natural transformations
\begin{equation*}
\mu :F\circ F\rightarrow F\qquad \text{and}\qquad \iota :\text{Id}_{%
\boldsymbol{A}}\rightarrow F.
\end{equation*}%
The commutativity of the diagrams in Figure~\ref{ec: M-categorie} is
equivalent in this case with the fact that $\left( F,\mu ,\iota \right) $ is
a \emph{monad}, see \cite{Be} for the definition of monads. In conclusion,
monads are in one-to-one correspondence to $[\boldsymbol{A},\boldsymbol{A}]$%
-categories with one object.
\end{fact}

\begin{fact}[The category $\boldsymbol{Opmon(M).}$]
Let $\left( \boldsymbol{M},\otimes ,1\right) $ be a monoidal category. An
opmonoidal functor is a triple $\left( F,\delta ,\varepsilon \right) $ that
consists of

\begin{enumerate}
\item A functor $F:\boldsymbol{M}\rightarrow \boldsymbol{M}$.

\item A natural transformation $\delta :=\left\{ \delta _{x,y}\right\}
_{(x,y)\in {M}_{0}\times {M}_{0}}$, with $\delta _{x,y}:F\left( x\otimes
y\right) \rightarrow F\left( x)\otimes F(y\right) .$

\item A map $\varepsilon :F(\mathbf{1})\rightarrow \mathbf{1}$ in $%
\boldsymbol{M}.$
\end{enumerate}

\noindent In addition, the transformations $\delta $ and $\varepsilon $ are
assumed to render commutative the diagrams in Figure~\ref{fig:opmon}. An
opmonoidal transformation $\alpha :(F,\delta ,\varepsilon )\rightarrow
(F^{\prime },\delta ^{\prime },\varepsilon ^{\prime })$ is a natural map $%
\alpha :F\rightarrow F^{\prime }$ such that, for arbitrary objects $x$ and $%
y $ in $\boldsymbol{M},$
\begin{equation*}
(\alpha _{x}\otimes \alpha _{y})\circ \delta _{x,y}=\delta _{x,y}^{\prime
}\circ \alpha _{x\otimes y}\qquad \text{and\qquad }\varepsilon ^{\prime
}\circ \alpha _{\mathbf{1}}=\varepsilon .
\end{equation*}%
Obviously the composition of two opmonoidal transformations is opmonoidal,
and the identity of an opmonoidal functor is an opmonoidal transformation.
The resulting category will be denoted by $\boldsymbol{Opmon(M)}.$ For two
opmonoidal functors $(F,\delta ,\varepsilon )$ and $(F^{\prime },\delta
^{\prime },\varepsilon ^{\prime })$ one defines
\begin{equation*}
(F,\delta ,\varepsilon )\otimes (F^{\prime },\delta ^{\prime },\varepsilon
^{\prime }):=\left( F\circ F^{\prime },\delta _{F^{\prime },F^{\prime
}}\circ F(\delta ^{\prime }),\varepsilon \circ F(\varepsilon ^{\prime
})\right) ,
\end{equation*}%
where $\delta _{F^{\prime },F^{\prime }}=\left\{ \delta _{F^{\prime
}(x),F^{\prime }(y)}\right\} _{x,y\in M_{0}}$. On the other hand, if $\mu
:F\rightarrow G$ and $\mu ^{\prime \prime }\rightarrow G^{\prime }$ are
opmonoidal transformations, then $\mu \otimes \mu ^{\prime }:=\mu G^{\prime
}\circ F\mu ^{\prime }$ is opmonoidal too. One can see easily that $\otimes $
defines a monoidal structure on $\boldsymbol{Opmon(M)}$ with unit object $(%
\text{Id}_{\boldsymbol{M}},\left\{ \text{Id}_{x\otimes y}\right\} _{x,y\in
M_{0}},\text{Id}_{\mathbf{1}}).$
\begin{figure}[tbh]
\begin{equation*}
\xymatrix{ \sv{F(x\otimes y\otimes z) \ar[rr]^-{\delta_{x\otimes
y,z}}}\ar[d]|-{\sa{\delta_{x, y\otimes z}}} & & \sv{F(x\otimes y)\otimes
F(z)} \ar[d]|-{\sa{\delta_{x, y}\otimes F(z)}}\\ \sv{F(x)\otimes F(y\otimes
z)}\ar[rr]_-{\sa{F(x)\otimes \delta_{y,z}}} & & \sv{F(x)\otimes F(y)\otimes
F(z) }}\qquad \xymatrix{ \sv{F(x)\otimes F(\mathbf{1})}
\ar[r]^-{\sa{F(x)\otimes\varepsilon_{\mathbf{1}}}} & \sv{F(x)} &
\sv{F(\mathbf{1})\otimes F(x)}\ar[l]_-{\sa{\varepsilon_{\mathbf{1}}\otimes
F(x)}}\\ & \sv{F(x)}\ar[ul]^-{\delta_{x,\mathbf{1}}}
\ar[ur]_-{\sa{\delta_{\mathbf{1},x}}}\ar@{=}[u] & }
\end{equation*}%
\caption{The definition of opmonoidal functors.}
\label{fig:opmon}
\end{figure}
\end{fact}

\begin{fact}[The categories $\boldsymbol{Alg(M)}$ and $\boldsymbol{Coalg(M)}%
. $]
Let $\left( \boldsymbol{M},\otimes ,\mathbf{1},\chi\right) $ be a braided
monoidal category with braiding $\chi:=\left\{ \chi _{x,y}\right\} _{\left(
x,y\right) \in {M}_{0}\times {M}_{0}}$, where $\chi_{x,y}:x\otimes
y\rightarrow y\otimes x.$ The category $\boldsymbol{Alg(M)}$ of all algebras
in $\boldsymbol{M}$ is monoidal too. Recall that an algebra in $\boldsymbol{M%
}$ is an $\boldsymbol{M}$-category with one object. As in \S \ref{Cat[A,A]},
such a category is uniquely determined by an object $X$ in $\boldsymbol{M}$
and two morphisms $m:X\otimes X\rightarrow X$ (the multiplication) and $u:%
\mathbf{1}\rightarrow X$ (the unit). The commutativity of the diagrams in
Figure~\ref{ec: M-categorie} means that the algebra is associative and
unital. If $\left( X,m,u\right) $ and $\left( X^{\prime },m^{\prime
},u^{\prime }\right) $ are algebras in $\boldsymbol{M},$ then $X\otimes
X^{\prime }$ is an algebra in $\boldsymbol{M}$ with multiplication%
\begin{equation*}
\left( m\otimes m^{\prime }\right) \circ \left( X\otimes \chi_{X^{\prime
},X}\otimes X^{\prime }\right) :\left( X\otimes X^{\prime }\right) \otimes
\left( X\otimes X^{\prime }\right) \rightarrow X\otimes X^{\prime }
\end{equation*}%
and unit $u\otimes u^{\prime }:\mathbf{1}\rightarrow X\otimes X^{\prime }$.

The monoidal category $\boldsymbol{Coalg(M)}$ of coalgebras in $\boldsymbol{M%
}$ can be defined in a similar way. Alternatively, one may take $\boldsymbol{%
Coalg(M)}:=\boldsymbol{Alg(M^{o})^{o}}.$ Note that the monoidal category of
coalgebras in $\boldsymbol{M}$ and the monoidal category of algebras in $%
\boldsymbol{M}^{o}$ are opposite each other.

It is not hard to see that $\boldsymbol{Coalg(M)}$ is $S$-distributive,
provided that $\boldsymbol{M}$ is so.
\end{fact}

\section{Factorizable $\boldsymbol{M}$-categories and
twisting systems.}

In this section we define factorizable $\boldsymbol{M}$-categories and
twisting systems. We shall prove that to every factorizable system
corresponds a certain twisting system. Under a mild extra assumption on the
monoidal category $\boldsymbol{M}$, we shall also produce enriched
categories using a special class of twisting systems that we call simple.

Throughout this section $S$ denotes a fixed set. We assume that all $%
\boldsymbol{M}$-categories that we work with are small, and that their set
of objects is $S$.

\begin{fact}[Factorizable $\boldsymbol{M}$-categories.]
Let $\boldsymbol{C}$ be a small enriched category over $(\boldsymbol{M}%
,\otimes ,\mathbf{1})$. We assume that $\boldsymbol{M}$ is $S$-distributive.
Suppose that $\boldsymbol{A}$ and $\boldsymbol{B}$ are $\boldsymbol{M}$%
-subcategories of $\boldsymbol{C}.$ Note that, by assumption, $%
A_{0}=B_{0}=C_{0}=S.$ For $x,$ $y$ and $u$ in $S$ we define
\begin{equation}
_{x}\varphi _{y}^{u}:{}_{x}{A}{}_{u}{B}_{y}\rightarrow {}_{x}{C}_{y},\qquad
{}_{x}\varphi _{y}^{u}:={}_{x}c_{y}^{u}\circ {}_{x}\alpha _{u}\beta _{y},
\label{ec:phi^u}
\end{equation}%
where $\boldsymbol{\alpha }:\boldsymbol{A}\rightarrow \boldsymbol{C}$ and $%
\boldsymbol{\beta }:\boldsymbol{B}\rightarrow \boldsymbol{C}$ denote the
corresponding inclusion $\boldsymbol{M}$-functors. By the universal property
of coproducts, for every $x$ and $y$ in $S,$ there is $_{x}\varphi
_{y}:{}_{x}A_{\overline{u}}B{_{y}\rightarrow {}_{x}{C}_{y}}$ such that
\begin{equation}
_{x}\varphi _{y}\circ {}_{x}\sigma _{y}^{u}={}_{x}\varphi _{y}^{u},
\label{ec:phi}
\end{equation}%
where $_{x}\sigma _{y}^{u}$ is the canonical inclusion of $_{x}A_{u}B_{y}$
into ${}_{x}A_{\overline{u}}B_{y}.$ Note that by the universal property of
coproducts $_{x}\varphi _{y}={}_{x}c_{y}^{\overline{u}}\circ {}_{x}\alpha _{%
\overline{u}}\beta _{y},$ as we have $_{x}\alpha _{\overline{u}}\beta
_{y}\circ {}_{x}\sigma _{y}^{u}={}_{x}\tau _{y}^{u}\circ {}_{x}\alpha
_{u}\beta _{y}$ and $_{x}c_{y}^{\overline{u}}\circ {}_{x}\tau
_{y}^{u}={}_{x}c_{y}^{u},$ where $_{x}\tau _{y}^{u}$ denotes the inclusion
of $_{x}C_{u}C_{y}$ into ${}_{x}C_{\overline{u}}C_{y}.$

We shall say that $\boldsymbol{C}$ \emph{factorizes} through $\boldsymbol{A}$
and $\boldsymbol{B}$ if $_{x}\varphi _{y}$ is an isomorphism, for all $x$
and $y$ in $S$. By definition, an $\boldsymbol{M}$-category $\boldsymbol{C}$
is \emph{factorizable }if it factorizes through $\boldsymbol{A} $ and $%
\boldsymbol{B},$ where $\boldsymbol{A}$ and $\boldsymbol{B}$ are certain $%
\boldsymbol{M}$-subcategories of $\boldsymbol{C}.$
\end{fact}

\begin{fact}[The twisting system associated to a factorizable $\boldsymbol{M}
$-category.]
Let $\boldsymbol{C}$ be an enriched category over a monoidal category $(%
\boldsymbol{M},\otimes ,\boldsymbol{1})$. We assume that $\boldsymbol{M}$ is
$S$-distributive. The family ${R}:=\{_{x}R_{z}^{y}\}_{x,y,z\in S}$ of
morphisms $_{x}R_{z}^{y}:{}_{x}{B}_{y}A_{z}\rightarrow {}_{x}A_{\overline{u}%
}B_{y}$ is called a \emph{twisting system} if the four diagrams in Figure %
\ref{fig:R_1} are commutative for all $x,$ $y,$ $z$ and $t$ in $S$.
\begin{figure}[tbh]
\centering
$%
\xymatrix{
\sv{_xB_yB_zA_t}
        \ar[d]|{\sa{RI\circ IR}}
        \ar[r]^-{\sa{bI}}
    & \sv{_xB_zA_t}
        \ar[d]|{R}\\
\sv{_xA_{\overline v}B_{\overline u}B_t}
        \ar@/_0px/[r]_-{\sa{Ib}}
    & \sv{_xA_{\overline v}B_t} }\quad
\xymatrix{
\sv{_xB_yA_zA_t}
        \ar[d]|{\sa{IR\circ RI}}
        \ar[r]^-{\sa{Ia}}
    & \sv{_xB_yA_t}
        \ar[d]|{R}\\
\sv{_xA_{\overline v}A_{\overline u}B_t}
        \ar@/_0px/[r]_-{\sa{aI}}
    & \sv{_xA_{\overline u}B_t} } \quad
\xymatrix{
\sv{{}_{x}A_{y}}
        \ar[r]^-{\sa{1^BI}}
        \ar[d]|{\sa{I 1^B}}
    &\sv{_xB_{x}A_{y}}
        \ar[d]|{\sa{R}}\\
\sv{_xA_{y}B_{y}}
        \ar[r]_-{\sa{\sigma}}
    &   \sv{_xA_{\overline u}B_y}}\quad
\xymatrix{
\sv{{}_{x}B_{y}}
        \ar[r]^-{\sa{I 1^A}}
        \ar[d]|{\sa{1^AI}}
    &\sv{_xB_yA_y}\ar[d]|{\sa{R}}\\
\sv{_xA_xB_y}\ar[r]_-{\sa{\sigma}}
    &\sv{_xA_{\overline u}B_y}} $%
\caption{The definition of twisting systems.}
\label{fig:R_1}
\end{figure}
\newline
Let us briefly explain the notation that we used in these diagrams. As a
general rule, we omit all subscripts and superscripts denoting elements in $%
S $, and which are attached to a morphism. The symbol $\otimes $ is also
omitted. For example, $a$ and $1^{A}$ (respectively $b$ and $1^{B}$) stand
for the suitable composition maps and identity morphisms in $\boldsymbol{A}$
(respectively $\boldsymbol{B}$). The identity morphism of an object in $%
\boldsymbol{M}$ is denoted by $I$. Thus, by $Ia:{}_{x}B_{y}A_{z}A_{t}%
\rightarrow {}_{x}B_{y}A_{t}$ we mean $_{x}B_{y}\otimes {}_{y}a_{t}^{z}$. On
the other hand, $aI:{}_{x}A_{\overline{v}}A_{\overline{u}}B_{t}\rightarrow
{}_{x}A_{\overline{u}}B_{t}$ is a shorthand notation for $_{x}a_{\overline{u}%
}^{\overline{v}}B_{t}$, which in turn is the unique map induced by $\left\{
_{x}\sigma _{t}^{u}\circ {}_{x}a_{u}^{v}B_{t}\right\} _{u,v\in S}.$ We shall
keep the foregoing notation in all diagrams that we shall work with.

We claim that to every factorizable $\boldsymbol{M}$-category $\boldsymbol{C}
$ corresponds a certain twisting system. By definition, the map $_{x}\varphi
_{y}$ constructed in (\ref{ec:phi}) is invertible for all $x$ and $y$ in $S.$
Let $_{x}\psi _{y}$ denote the inverse of $_{x}\varphi _{y}$. For $x$, $y$
and $z$ in $S$, we can now define
\begin{equation}
_{x}R_{z}^{y}:{}_{x}B_{y}A_{z}\rightarrow {}_{x}A_{\overline{u}}B_{z},\quad
_{x}R_{z}^{y}:={}_{x}\psi _{z}\circ {}_{x}c_{z}^{y}\circ {}_{x}\beta
_{y}\alpha _{z}.  \label{ec:R}
\end{equation}
\end{fact}

\begin{theorem}
\label{teo1} If $\boldsymbol{C}$ is a factorizable enriched category over an
$S$-distributive monoidal category $\boldsymbol{M}$, then the maps in (\ref%
{ec:R}) define a twisting system.
\end{theorem}

\begin{proof}
Let us first prove that the first diagram in Figure \ref{fig:R_1} is
commutative. We fix $x,$ $y,$ $z$ and $t$ in $S,$ and we consider the
following diagram.
\begin{equation*}
\xymatrix{\ar@{}[drr] |*+[o][F-]{A} \sv{_xB_yA_{\overline u}B_t}
\ar[rr]^{\sa{\beta{}\alpha{}\beta{}}} \ar[d]|{\sa{\beta{}\alpha{}I}} &
\ar@{}[d]{} & \sv{_xC_yC_{\overline u}C_t} \ar@{}[ddddr]|-*+[o][F-]{F}
\ar[r]^{\sa{Ic}} \ar@{=}[d]{} & \sv{_xC_yC_t} \ar[dddd]|{c}\\
\sv{_xC_yC_{\overline u}B_t} \ar[rr]|{\sa{II\beta{}}} \ar@{}[drr]
|*+[o][F-]{B} \ar[d]|{\sa{cI}} & \ar@{}[d]{} & \sv{_xC_yC_{\overline u}C_t}
\ar[d]|{\sa{cI}} & \\ \sv{_xC_{\overline u}B_t} \ar@{}[drr]|*+[o][F-]{C}
\ar[rr]|{\sa{I\beta{}}} \ar[d]|<(.25){\sa{\psi I}} & \ar@{}[d]{} &
\sv{_xC_{\overline u}C_t} \ar@{=}[d]{} & \\ \sv{_xA_{\overline
v}B_{\overline u}B_t} \ar@{}[dr] |*+[o][F-]{D}
\ar@/^0px/[r]^-{\sa{\alpha{}\beta{}\beta{}}} \ar[d]|{\sa{Ib}} &
\sv{_xC_{\overline v}C_{\overline u}C_t} \ar@{}[dr] |*+[o][F-]{E}
\ar@/^0px/[r]|-{\sa{cI}} \ar[d]|{\sa{Ic}} & \sv{_xC_{\overline u}C_t}
\ar[d]|{\sa{c}} & \\ \sv{_xA_{\overline v}B_t} \ar[r]_{\sa{\alpha{}\beta{}}}
& \sv{_xC_{\overline v}C_t} \ar[r]_{\sa{c}} & \sv{_xC_t} \ar@{=}[r]{} &
\sv{_xC_t} }
\end{equation*}%
Since the tensor product in a monoidal category is a functor, that is in
view of \eqref{ec:TP=functor}, we have
\begin{equation}
_{x}C_{y}C_{u}\beta _{t}\circ {}_{x}\beta _{y}\alpha _{u}B_{t}={}_{x}\beta
_{y}\alpha _{u}\beta _{t},  \label{eq:ab}
\end{equation}%
for any $u$ in $S.$ Hence by the universal property of the coproduct and the
construction of the maps $_{x}C_{y}C_{\overline{u}}\beta _{t},{}$ $_{x}\beta
_{y}\alpha _{\overline{u}}B_{t}$ and ${}_{x}\beta _{y}\alpha _{\overline{u}%
}\beta _{t}$ we deduce that the relation which is obtained by replacing $u$
with $\overline{u}$ in (\ref{eq:ab}) holds true. This means that the square
(A) is commutative. Proceeding similarly one shows that (B) is commutative
as well. Furthermore, $_{x}c_{\overline{u}}^{\overline{v}}C_{t},$ $%
_{x}\alpha _{\overline{v}}\beta _{\overline{u}}\beta _{t}$ and $_{x}\psi _{%
\overline{u}}B_{t}$ are induced by $\left\{ _{x}c_{u}^{\overline{v}}\otimes
{}_{u}C_{t}\right\} _{u\in S},$ $\left\{ _{x}\alpha _{\overline{v}}\beta
_{u}\otimes {}_{u}\beta _{t}\right\} _{u\in S}$ and $\left\{ _{x}\psi
_{u}\otimes {}_{u}B_{t}\right\} _{u\in S},$ respectively. Hence their
composite $\lambda :={}_{x}c_{\overline{u}}^{\overline{v}}C_{t}\circ
{}_{x}\alpha _{\overline{v}}\beta _{\overline{u}}\beta _{t}\circ {}_{x}\psi
_{\overline{u}}B_{t}$ is induced by $\left\{ \lambda _{u}\right\} _{u\in S},$
where%
\begin{equation*}
\lambda _{u}=\left( _{x}c_{u}^{\overline{v}}\otimes {}_{u}C_{t}\right) \circ
\left( _{x}\alpha _{\overline{v}}\beta _{u}\otimes {}_{u}\beta _{t}\right)
\circ \left( _{x}\psi _{u}\otimes {}_{u}B_{t}\right) =\left( _{x}c_{u}^{%
\overline{v}}\circ {}_{x}\alpha _{\overline{v}}\beta _{u}\circ {}_{x}\psi
_{u}\right) \otimes {}_{u}\beta _{t}=\left( _{x}\varphi _{u}\circ {}_{x}\psi
_{u}\right) \otimes {}_{u}\beta _{t}.
\end{equation*}%
Since $_{x}\psi _{u}$ is the inverse of $_{x}\varphi _{u}$ it follows that $%
\lambda _{u}={}_{x}C_{u}\beta _{t},$ for every $u\in S.$ In conclusion
\begin{equation*}
_{x}c_{\overline{u}}^{\overline{v}}C_{t}\circ {}_{x}\alpha _{\overline{v}%
}\beta _{\overline{u}}\beta _{t}\circ {}_{x}\psi _{\overline{u}%
}B_{t}={}_{x}C_{\overline{u}}\beta _{t},
\end{equation*}%
so (C) is a commutative square. Since $\boldsymbol{\beta }$ is an $%
\boldsymbol{M}$-functor it follows that $\left\{ _{x}C_{v}c_{t}^{u}\circ
{}_{x}\alpha {}_{v}\beta _{u}\beta _{t}{}\right\} _{u,v\in S}$ and $\left\{
_{x}\alpha {}_{v}\beta _{t}\circ {}{}_{x}A_{v}b_{t}^{u}\right\} _{u,v\in S}$
are equal. Therefore these families induce the same morphism, that is
\begin{equation*}
_{x}C_{\overline{v}}c_{t}^{\overline{u}}\circ {}_{x}\alpha {}_{\overline{v}%
}\beta _{\overline{u}}\beta _{t}{}={}_{x}\alpha {}_{\overline{v}}\beta
_{t}\circ {}{}_{x}A_{\overline{v}}b_{t}^{\overline{u}}.
\end{equation*}%
Hence (D) is commutative too. Since the composition of morphisms in $%
\boldsymbol{C}$ is associative, we have
\begin{equation*}
_{x}c_{t}^{\overline{v}}\circ {}_{x}C_{\overline{v}}c_{t}^{\overline{u}%
}={}_{x}c_{t}^{\overline{u}}\circ {}_{x}c_{\overline{u}}^{\overline{v}%
}C_{t}\qquad \text{and\qquad }_{x}c_{t}^{y}\circ {}_{x}C_{y}c_{t}^{\overline{%
u}}={}_{x}c_{t}^{\overline{u}}\circ {}_{x}c_{\overline{u}}^{y}C_{t}.
\end{equation*}%
These equations imply that (E) and (F)\ are commutative. Summarizing, we
have just proved that all diagrams (A)-(F) are commutative. By diagram
chasing it results that the outer square is commutative as well, that is%
\begin{equation*}
{}_{x}\varphi _{t}\circ {}_{x}A_{\overline{v}}{b}_{t}^{\overline{u}}\circ
{}_{x}R_{\overline{u}}^{y}B_{t}={}_{x}c_{t}^{y}\circ {}_{x}\beta _{y}\varphi
_{t}.
\end{equation*}%
Left composing and right composing both sides of this equation by $_{x}\psi
_{t}$ and ${}_{x}B_{y}R_{t}^{z},$ respectively, yield%
\begin{align*}
{}_{x}A_{\overline{v}}b_{t}^{\overline{u}}\circ {}_{x}R_{\overline{u}%
}^{y}b_{t}\circ {}_{x}B_{y}R_{t}^{z}& ={}_{x}\psi _{t}\circ
{}_{x}c_{t}^{y}\circ {}_{x}\beta _{y}\varphi _{t}\circ {}_{x}B_{y}R_{t}^{z}
\\
& ={}_{x}\psi _{t}\circ {}_{x}c_{t}^{y}\circ {}_{x}\beta _{y}\varphi
_{t}\circ {}_{x}B_{y}\psi _{t}\circ {}_{x}B_{y}c_{t}^{z}\circ
{}_{x}B_{y}\beta _{z}\alpha _{t} \\
& ={}_{x}\psi _{t}\circ {}_{x}c_{t}^{y}\circ {}_{x}C_{y}{}c_{t}^{z}\circ
{}_{x}\beta _{y}\beta _{z}\alpha _{t},
\end{align*}%
where for the second and third relations we used the definition of $%
_{y}R_{t}^{z}$ and that $_{y}\varphi _{t}$ and $_{y}\psi _{t}$ are inverses
each other. On the other hand, the definition of ${}_{x}R_{t}^{z},$ the fact
that $\boldsymbol{\beta }$ is a functor and associativity of the composition
in $\boldsymbol{C\ }$imply the following sequence of identities
\begin{align*}
{}_{x}R_{t}^{z}\circ {}{}_{x}b_{z}^{y}A_{t}& ={}{}_{x}\psi _{t}\circ
{}_{x}c_{t}^{z}\circ {}_{x}\beta _{z}\alpha _{t}\circ {}{}_{x}b_{z}^{y}A_{t}
\\
& ={}_{x}\psi _{t}\circ {}_{x}c_{t}^{z}\circ {}_{x}c_{z}^{y}C_{t}\circ
{}{}_{x}\beta _{y}\beta _{z}\alpha _{t} \\
& ={}{}_{x}\psi _{t}\circ {}_{x}c_{t}^{y}\circ {}_{x}C_{y}{}c_{t}^{z}\circ
{}_{x}\beta _{y}\beta _{z}a_{t}.
\end{align*}%
In conclusion, the first diagram in Figure \ref{fig:R_1} is commutative.
Taking into account the definition of $_{x}R_{y}^{x},$ the identity $%
_{x}\beta _{x}\circ 1_{x}^{B}=1_{x}$ and the compatibility relation between
the composition and the identity morphisms in an enriched category, we get
the following sequence of equations%
\begin{equation*}
_{x}\varphi _{y}{}\circ {}_{x}R_{y}^{x}\circ 1_{x}^{B}A_{y}={}_{x}\varphi
_{y}{}\circ {}_{x}\psi _{y}{}\circ {}_{x}c_{y}^{x}{}\circ {}_{x}\beta
_{x}\alpha _{y}\circ 1_{x}^{B}A_{y}={}_{x}c_{y}^{x}{}\circ {}1_{x}\alpha
_{y}={}_{x}\alpha _{y}.
\end{equation*}%
Analogously, using the definition of $_{x}\varphi _{y}{}$ and the properties
of identity morphisms, we get
\begin{equation*}
_{x}\varphi _{y}{}\circ {}_{x}\sigma _{y}^{y}\circ
{}_{x}A_{y}1_{y}^{B}={}_{x}\varphi _{y}^{y}{}\circ
{}{}_{x}A_{y}1_{y}^{B}={}_{x}c_{y}^{y}{}\circ {}{}_{x}\alpha _{y}\beta
_{y}\circ {}_{x}A_{y}1_{y}^{B}={}_{x}c_{y}^{y}{}\circ {}{}_{x}\alpha
_{y}1_{y}={}_{x}\alpha _{y}.
\end{equation*}%
Since $_{x}\varphi _{y}{}$ is an isomorphisms, in view of the above
computations, it follows that the third diagram is commutative as well. One
can prove in a similar way that the remaining two diagrams in Figure \ref%
{fig:R_1} are commutative.
\end{proof}

\begin{fact}
\label{fa:R^tilda}We have noticed in the introduction that to every twisting
system of groups (or, equivalently, every matched pair of groups) one
associates a factorizable group. Trying to prove a similar result for a
twisting system ${R}$ between the $\boldsymbol{M}$-categories $\boldsymbol{B}
$ and $\boldsymbol{A}$ we have encountered some difficulties due to the fact
that, in general, the image of the map
\begin{equation*}
_{x}R_{z}^{y}:{}_{x}B_{y}A_{z}\rightarrow \textstyle\bigoplus\limits_{u\in
S}\,_{x}A_{u}B_{z}
\end{equation*}%
is not included into a summand $_{x}A_{u}B_{z},$ for some $u\in S$ that
depends on $x,$ $y$ and $z.$ For this reason, in this paper we shall
investigate only those twisting systems for which there are a function $%
|\cdots |:S^{3}\rightarrow S$ and the maps $_{x}\widetilde{R}%
_{z}^{y}:{}_{x}B_{y}A_{z}\rightarrow {}_{x}A_{|xyz|}B_{z}$ such that
\begin{equation}
_{x}R_{z}^{y}={}_{x}\sigma _{z}^{|xyz|}\circ {}_{x}\widetilde{R}_{z}^{y},
\label{R simplu}
\end{equation}%
for all $x,y,z,\in S.$ For them we shall use the notation $(\widetilde{{R}}%
,|\cdots |).$
\end{fact}

\begin{proposition}
\label{thm R simplu} Let $\boldsymbol{M}$ be a monoidal category which is $S$%
-distributive. Let $|\cdots |:S^{3}\rightarrow S$ and $\{_{x}\widetilde{R}%
_{z}^{y}\}_{x,y,z\in S}$ be a function and a set of maps as above. The
family $\{_{x}R_{z}^{y}\}_{x,y,z\in S}$ defined by (\ref{R simplu}) is a
twisting system if and only if, for any $x,y,z,t\in S,$ the following
relations hold:
\begin{gather}
_{x}\sigma _{t}^{|xy|yzt||}\circ {}_{x}A_{|xy|yzt||}b_{t}^{|yzt|}\circ {}_{x}%
\widetilde{R}_{|yzt|}^{y}B_{t}{}\circ {}_{x}B_{y}\widetilde{R}%
_{t}^{z}={}_{x}\sigma _{t}^{|xzt|}\circ {}{}_{x}\widetilde{R}_{t}^{z}\circ
{}_{x}b_{z}^{y}A_{t},  \label{Rs-1} \\
_{x}\sigma _{t}^{|xyz|zt||}\circ {}_{x}a_{||xyz|zt|}^{|xyz|}B_{t}\circ
{}_{x}A_{|xyz|}\widetilde{R}_{t}^{z}\circ {}_{x}\widetilde{R}%
_{z}^{y}A_{t}={}_{x}\sigma _{t}^{|xyt|}\circ {}_{x}\widetilde{R}_{t}^{y}%
\circ{}_{x}B_{y}a_{t}^{z},  \label{Rs-2} \\
{}_{x}\sigma _{y}^{|xxy|}\circ {\ }_{x}\widetilde{R}_{y}^{x}\circ
(1_{x}^{B}\otimes {}_{x}A_{y})={}_{x}\sigma _{y}^{y}\circ(_{x}A_{y}\otimes
1_{y}^{B}),  \label{Rs-3} \\
_{x}\sigma _{y}^{|xyy|}\circ {}_{x}\widetilde{R}_{y}^{y}\circ
(_{x}B_{y}\otimes 1_{y}^{A})={}_{x}\sigma _{y}^{x}\circ (1_{x}^{A}\otimes
{}_{x}B_{y}).  \label{Rs-4}
\end{gather}
\end{proposition}

\begin{proof}
We claim that $\{_{x}\widetilde{R}_{z}^{y}\}_{x,y,z\in S}$ satisfy (\ref%
{Rs-1}) if and only if $\{_{x}{R}_{z}^{y}\}_{x,y,z\in S}$ render commutative
the first diagram in \ Figure \ref{fig:R_1}. Indeed, let us consider the following diagram.
\begin{equation*}
\xymatrix{ \sv{_xB_yB_zA_t}\ar[r]^-{I\widetilde{R}}\ar[d]_-{bI}
&\sv{_xB_yA_{|yzt|}B_t}\ar[r]^-{I\sigma}\ar[d]_-{\widetilde{R}I}
&\sv{_xB_yA_{\overline{u}}B_t}\ar[d]^-{{R}I}\ar@{}[dl]^>(.30)*+[o][F-]{B} \\
\sv{_xB_zA_t}\ar[dd]_-{\widetilde{R}}
&\sv{_xA_{|xy|yzt||}B_{|yzt|}B_t}\ar[r]^-{\sigma I}\ar[d]_-{Ib}
&\sv{_xA_{\overline{v}}B_{\overline{u}}B_t}
\ar[d]^-{Ib}\ar@{}[dl]^>(.30)*+[o][F-]{C} \\ \ar@{}[dr]^>(.40)*+[o][F-]{A}
&\sv{_xA_{|xy|yzt||}B_t}\ar[r]^-{\sigma}
&\sv{_xA_{\overline{v}}B_t}\ar@{=}[d] \\
\sv{_xA_{|xzt|}B_t}\ar[rr]_-{\sigma} & &\sv{_xA_{\overline{v}}B_t} \\ }
\end{equation*}%
The squares (B) and (C) are commutative by the definition of $_{x}R_{%
\overline{u}}^{y}:{}_{x}B_{y}A_{\overline{u}}\rightarrow {}_{x}A_{\overline{v%
}}B_{\overline{u}}$ and $_{\overline{v}}b_{t}^{\overline{u}}:{}_{\overline{v}%
}B_{\overline{u}}B_{t}\rightarrow {}_{\overline{v}}B_{t}$. Hence the hexagon
(A) is commutative if and only if the outer square is commutative. This
proves our claim as (A) and the outer square in Figure \ref{fig:R_1} are
commutative if and only if \eqref{Rs-1} holds and the first diagram in
Figure \ref{fig:R_1} is commutative, respectively. Similarly one shows that
the commutativity of the second diagram from Figure \ref{fig:R_1} is
equivalent to \eqref{Rs-2}. On the other hand, obviously, the third and
fourth diagrams in Figure \ref{fig:R_1} are commutative if and only if %
\eqref{Rs-3} and \eqref{Rs-4} hold, so the proposition is proved.
\end{proof}

The inclusion maps make difficult to handle the equations (\ref{Rs-1})-(\ref%
{Rs-4}). In some cases we can remove these morphisms by imposing more
conditions on the map $|\cdots |$ or on the monoidal category $\boldsymbol{M}
$.

\begin{fact}[The assumption (\dag ).]
\label{fa:dag} Let $\boldsymbol{M}$ be a monoidal category which is $S$%
-distributive. We shall say that $\boldsymbol{M}$ satisfies the hypothesis ($%
\dag $) if for any coproduct $(\textstyle\bigoplus {}_{i\in S}X_{i},\{\sigma
_{i}\}_{i\in S})$ in $\boldsymbol{M}$ and any morphisms $f^{\prime
}:X\rightarrow X_{i^{\prime }}$ and $f^{\prime \prime }:X\rightarrow
X_{i^{\prime \prime }}$ such that $\sigma _{i^{\prime }}\circ f^{\prime
}=\sigma _{i^{\prime \prime }}\circ f^{\prime \prime }$, then either $X$ is
an initial object $\emptyset $ in $\boldsymbol{M},$ or $f^{\prime
}=f^{\prime \prime }$ and $i^{\prime }=i^{\prime \prime }$.

The prototype for the class of monoidal categories that satisfy the
condition (\dag ) is $\boldsymbol{Set}$. Indeed, let $\{X_{i}\}_{i\in S}$ be
a family of sets, and let $\sigma _{i}$ denote the inclusion of $X_{i}$ into
the disjoint union $\textstyle\coprod_{i\in S}X_{i}.$ We assume that $%
f^{\prime }:X\rightarrow X_{i^{\prime }}$ and $f^{\prime \prime
}:X\rightarrow X_{i^{\prime \prime }}$ are functions such that $X$ in not
the empty set, the initial object of $\boldsymbol{Set}$, and $\sigma
_{i^{\prime }}\circ f^{\prime }=\sigma _{i^{\prime \prime }}\circ f^{\prime
\prime }.$ Then in view of the computation
\begin{equation*}
(i^{\prime },f^{\prime }(x))=(\sigma _{i^{\prime }}\circ f^{\prime
})(x)=(\sigma _{i^{\prime \prime }}\circ f^{\prime \prime })(x)=(i^{\prime
\prime },f^{\prime \prime }(x))
\end{equation*}%
it follows that $f^{\prime }=f^{\prime \prime }$ and $i^{\prime }=i^{\prime
\prime }.$
\end{fact}

\begin{corollary}
\label{cor: R simplu1} Let $\boldsymbol{M}$ be an $S$-distributive monoidal
category. Let $\boldsymbol{A}$ and $\boldsymbol{B}$ be two $\boldsymbol{M}$%
-categories such that $A_{0}=B_{0}=S$. Given a function $\left\vert \cdots
\right\vert :S^{3}\rightarrow S$ and the maps $\{_{x}\widetilde{{R}}%
_{z}^{y}\}_{x,y,z\in S}$ as in \S \ref{fa:R^tilda}, let us consider the
following four conditions:

\begin{enumerate}
\item[(i)] If $_{x}B_{y}B_{z}A_{t}\ $is not an initial object, then $%
|xy|yzt||=|xzt|$ and
\begin{equation}
_{x}A_{|xzt|}b_{t}^{|yzt|}\circ {}_{x}\widetilde{R}_{|yzt|}^{y}B_{t}{}\circ
{}_{x}B_{y}\widetilde{R}_{t}^{z}={}_{x}\widetilde{R}_{t}^{z}\circ
{}_{x}b_{z}^{y}A_{t};  \label{cor1-1}
\end{equation}

\item[(ii)] If $_{x}B_{y}A_{z}A_{t}\ \ $is not an initial object, then $%
||xyz|zt|=|xyt|$ and
\begin{equation}
_{x}a_{|xyt|}^{|xyz|}B_{t}\circ {}_{x}A_{|xyz|}\widetilde{R}_{t}^{z}\circ
{}_{x}\widetilde{R}_{z}^{y}A_{t}={}_{x}\widetilde{R}_{t}^{y}\circ
{}_{x}B_{y}a_{t}^{z};  \label{cor1-2}
\end{equation}

\item[(iii)] If $_{x}A_{y}$ $\ $is not an initial object, then $|xxy|=y$ and%
\begin{equation}
{\ }_{x}\widetilde{R}_{y}^{x}\circ (1_{x}^{B}\otimes
{}_{x}A_{y})={}_{x}A_{y}\otimes 1_{y}^{B};  \label{cor1-3}
\end{equation}

\item[(iv)] If $_{x}B_{y}$ $\ $is not an initial object, then $|xyy|=x$ and
\begin{equation}
_{x}\widetilde{R}_{y}^{y}\circ (_{x}B_{y}\otimes
1_{y}^{A})={}1_{x}^{A}\otimes {}_{x}B_{y}.  \label{cor1-4}
\end{equation}
\end{enumerate}

The above conditions imply the relations (\ref{Rs-1})-(\ref{Rs-4}). Under
the additional assumption that $\boldsymbol{M}$ satisfies the hypothesis
(\dag ), the reversed implication holds as well.
\end{corollary}

\begin{proof}
Let us prove that the condition (i) implies the relation (\ref{Rs-1}). In
the case when $_{x}B_{y}B_{z}A_{t}=\emptyset $ this is clear, as both sides
of (\ref{Rs-1}) are morphisms from an initial object to $_{x}A_{\overline{u}%
}B_{t}.$ Let us suppose that $_{x}B_{y}B_{z}A_{t}\neq \emptyset .$ By
composing both sides of (\ref{cor1-1}) with $_{x}\sigma
_{t}^{|xy|yzt||}={}_{x}\sigma _{t}^{|xzt|}$ we get the equation (\ref{Rs-1}%
). Similarly, the conditions (ii), (iii) and (iv) imply the relations (\ref%
{Rs-2}), (\ref{Rs-3}) and (\ref{Rs-4}), respectively.

Let us assume that $\boldsymbol{M}$ satisfies the hypothesis (\dag ). We
claim that (\ref{Rs-1}) implies (i). If $_{x}B{}_{y}B${}$_{z}A_{t}\ $is not
an initial object we take $f^{\prime }$ and $f^{\prime \prime }$ to be the
left hand side and the right hand side of (\ref{cor1-1}), respectively. We
also set $i^{\prime }:=|xy|yzt||$ and $i^{\prime \prime }:=|xzt|.$ In view
of (\dag ), it follows that $f^{\prime }=f^{\prime \prime }$ and $i^{\prime
}=i^{\prime \prime },$ so our claim has been proved. We conclude the proof
in the same way.
\end{proof}

\begin{fact}[$\mathbb{K}$-linear monoidal categories.]
\label{fa:ddag} Recall that $\boldsymbol{M}$ is $\mathbb{K}$-linear if its
hom-sets are $\mathbb{K}$-modules, and both the composition and the tensor
product of morphisms are $\mathbb{K}$-bilinear maps. For instance, $\mathbb{K%
}$-$\boldsymbol{Mod}$, $H$-$\boldsymbol{Mod}$, $\boldsymbol{Comod}$-$H$ and $%
\Lambda $-$\boldsymbol{Mod}$-$\Lambda $ are $S$-distributive linear monoidal
categories, for any set $S$.

Note that the (\dag ) condition fail in a $\mathbb{K}$-linear monoidal
category $\boldsymbol{M}$. Indeed let us pick up an object $X$, which is not
an initial object, and a coproduct $(\textstyle\bigoplus_{i\in
S}X_{i},\{\sigma _{i}\}_{i\in S})$ in $\boldsymbol{M}$. If $f^{\prime
}:X\rightarrow X_{i^{\prime }}$ and $f^{\prime \prime }:X\rightarrow
X_{i^{\prime \prime }}$ are the zero morphisms, then of course $\sigma
_{i^{\prime }}\circ f^{\prime }=\sigma _{i^{\prime \prime }}\circ f^{\prime
\prime },$ but neither $i^{\prime }=i^{\prime \prime }$ nor $f^{\prime
}=f^{\prime \prime }$, in general.

Nevertheless, the relations (\ref{Rs-1})-(\ref{Rs-4}) can also be simplified
if $\boldsymbol{M}$ is a linear monoidal category. For any coproduct $(%
\textstyle\bigoplus_{i\in S}X_{i},\{\sigma _{i}\}_{i\in S})$ in $\boldsymbol{%
M}$ and every $i\in S,$ there is a map $\pi _{i}:\textstyle\bigoplus_{i\in
S}X_{i}\rightarrow X_{i}$ such that $\pi _{i}\circ \sigma _{i}=X_{i}$ and $%
\pi _{i}\circ \sigma _{j}=0$, provided that $j\neq i$. Hence, supposing that
$f^{\prime }:X\rightarrow X_{i^{\prime }}$ and $f^{\prime \prime
}:X\rightarrow X_{i^{\prime \prime }}$ are morphisms such that $\sigma
_{i^{\prime }}\circ f^{\prime }=\sigma _{i^{\prime \prime }}\circ f^{\prime
\prime },$ we must have either $i^{\prime }=i^{\prime \prime }$ and $%
f^{\prime }=f^{\prime \prime },$ or $i^{\prime }\neq i^{\prime \prime }$ and
$f^{\prime }=0=f^{\prime \prime }.$
\end{fact}

Using the above property of linear monoidal categories, and proceeding as in
the proof of the previous corollary, we get the following result.

\begin{corollary}
\label{cor R simplu2} Let $\boldsymbol{M}$ be an $S$-distributive $\mathbb{K}
$-linear monoidal category. If $\boldsymbol{A}$ and $\boldsymbol{B}$ are $%
\boldsymbol{M}$-categories, then the relations (\ref{Rs-1})-(\ref{Rs-4}) are
equivalent to the following conditions:

\begin{enumerate}
\item[(i)] If $|xy|yzt||=|xzt|\ $then the relation (\ref{cor1-1}) holds;
otherwise, each side of this identity has to be the zero map;

\item[(ii)] If $||xyz|zt|=|xyt|,$ then the relation (\ref{cor1-2}) holds;
otherwise, each side of this identity has to be the zero map;

\item[(iii)] If $|xxy|=y,$ then the relation (\ref{cor1-3}) holds;
otherwise, each side of this identity has to be the zero map;

\item[(iv)] If $|xyy|=x,$ then the relation (\ref{cor1-4}) holds; otherwise,
each side of this identity has to be the zero map.
\end{enumerate}
\end{corollary}

\begin{fact}[Simple twisting systems.]
\label{STS} The proper context for constructing an enriched category $%
\boldsymbol{A}\otimes _{{R}}\boldsymbol{B}$ out of a special type of
twisting system ${R}$ is provided by Corollary \ref{cor: R simplu1}.

By definition, the couple $(\widetilde{{R}},|\cdots |)$ is a \emph{simple
twisting system between} $\boldsymbol{B}$ \emph{and} $\boldsymbol{A}$ if the
function $\left\vert \cdots \right\vert :S^{3}\rightarrow S$ and the maps $%
\{_{x}\widetilde{R}_{z}^{y}\}_{x,y,z\in S}$ as in \S \ref{fa:R^tilda}
satisfy the conditions (i)-(iv) in Corollary \ref{cor: R simplu1}. As a part
of the definition, we also assume that $_{x}A_{\left\vert xyz\right\vert
}B_{z}$ is not an initial object whenever $_{x}B_{y}A_{z}$ is not so.

The latter technical assumption will be used to prove the associativity of
the composition in ${\boldsymbol{A}\otimes _{{R}}\boldsymbol{B}},$ our
categorical version of the twisted tensor product of two algebras, which we
are going to define in the next subsection. Note that for $\boldsymbol{Set}$
this condition is superfluous (if the source of $_{x}\widetilde{R}_{z}^{y} $
is not empty, then its target cannot be the empty set).

For a simple twisting system $(\widetilde{{R}},|\cdots |)$ we define the
maps $_{x}R_{z}^{y}$ using the relation (\ref{R simplu}). By Corollary \ref%
{cor: R simplu1} and Proposition \ref{thm R simplu} it follows that ${R}%
:=\{_{x}R_{z}^{y}\}_{x,y,z\in S}$ is a twisting system.
\end{fact}

\begin{fact}[The category $\boldsymbol{A}\otimes _{{R}}\boldsymbol{B}$.]
\label{fa:ARB}For a simple twisting system $(\widetilde{{R}},|\cdots |)$ we
set
\begin{equation*}
\left( \boldsymbol{A}\otimes _{{R}}\boldsymbol{B}\right) _{0}:=S\quad \text{%
and\quad }{_{x}(\boldsymbol{A}\otimes _{{R}}\boldsymbol{B})}_{y}:=\textstyle%
\bigoplus\limits_{u\in S}{}_{x}A_{u}\otimes {}_{u}B_{y}={}_{x}A_{\overline{u}%
}B_{y}.
\end{equation*}%
Let us fix three elements $x,$ $y$ and $z$ in $S.$ By definition $_{x}A_{%
\overline{u}}B_{y}A_{\bar{v}}B_{z}:=\textstyle\bigoplus\limits_{u,v\in
S}{}\,_{x}A_{u}B_{y}A_{v}B_{z}$, and
\begin{equation*}
_{x}A_{\overline{u}}B_{y}A_{\overline{v}}B_{z}\cong {}_{x}A_{\overline{u}%
}B_{y}\otimes {}_{y}A_{\overline{v}}B_{z}
\end{equation*}%
as $\boldsymbol{M}$ is $S$-distributive. Via this identification, the
canonical inclusion of $_{x}A_{u}B_{y}A_{v}B_{z}$ into the coproduct $_{x}A_{%
\overline{u}}B_{y}A_{\overline{v}}B_{z}$ corresponds to $_{x}\sigma
_{y}^{u}\sigma _{z}^{v}={}_{x}\sigma _{y}^{u}\otimes {}_{y}\sigma _{z}^{v}.$
Thus, there is a unique morphism $_{x}c_{z}^{y}:{}_{x}A_{\overline{u}%
}B_{y}A_{\overline{v}}B_{z}\rightarrow {}_{x}A_{\overline{u}}B_{z}$ such
that
\begin{equation*}
_{x}c_{z}^{y}\circ {}_{x}\sigma _{y}^{u}\sigma _{z}^{v}={}_{x}\sigma
_{z}^{\left\vert uyv\right\vert }\circ {}_{x}a_{\left\vert uyv\right\vert
}b_{z}\circ {}_{x}A_{u}\widetilde{R}_{v}^{y}B_{z},
\end{equation*}%
for all $u,v\in S$. Finally, we set $1_{x}:={}_{x}\sigma _{x}^{x}\circ
(1_{x}^{A}\otimes 1_{x}^{B})$, and we define
\begin{equation*}
_{x}\alpha _{y}:={}_{x}\sigma _{y}^{y}\circ ({}_{x}A_{y}\otimes
1_{y}^{B})\qquad \text{and}\qquad _{x}\beta _{y}:={}_{x}\sigma _{y}^{x}\circ
(1_{x}^{A}\otimes {}_{x}B_{y}).
\end{equation*}
\end{fact}

\begin{fact}[Domains.]
To show that the above data define an enriched monoidal category $%
\boldsymbol{A}\otimes _{{R}}\boldsymbol{B}$ we need an extra hypothesis on $%
\boldsymbol{M}.$ By definition, a monoidal category $\boldsymbol{M}$ is a
\emph{domain }in the case when the tensor product of two objects$\ $is an
initial object if and only if at least one of them is an initial object. By
convention, a monoidal category that has no initial objects is a domain as
well.

Obviously $\boldsymbol{Set}$ is a domain. If $\mathbb{K}$ is a field, then $%
\mathbb{K}$-$\boldsymbol{Mod}$ is a domain. Keeping the assumption on $%
\mathbb{K},$ the categories $H$-$\boldsymbol{Mod}$ and $\boldsymbol{Comod}$-$%
H$ are domains, as their tensor product is induced by that one of $\mathbb{K}
$-$\boldsymbol{Mod}.$ On the other hand, if $\mathbb{K}$ is not a field,
then $\mathbb{K}$-$\boldsymbol{Mod}$ and $\Lambda $-$\boldsymbol{Mod}$-$%
\Lambda $ are not necessarily domains. For instance, $\mathbb{Z}$-$%
\boldsymbol{Mod}\cong \mathbb{Z}$-$\boldsymbol{Mod}$-$\mathbb{Z}$ is not a
domain.
\end{fact}

\begin{lemma}
\label{le:STS}Let $\boldsymbol{M}$ be an $S$-distributive monoidal domain.
Let $(\widetilde{{R}},\left\vert \dots \right\vert )$ denote a simple
twisting system between $\boldsymbol{B}$ and $\boldsymbol{A}.$

\begin{enumerate}
\item If $_{x}A_{u}B_{y}A_{v}B_{z}A_{w}B_{t}\neq \emptyset $ then $%
\left\vert uyq\right\vert =\left\vert pvq\right\vert =\left\vert
pzw\right\vert ,$ where $p=\left\vert uyv\right\vert $ and $q=\left\vert
vzw\right\vert .$

\item In the following diagram all squares are well defined
and commutative.
\begin{equation*}
\xymatrix{ \ar@{}[dr]^>(.60)*+[o][F-]{F} \sv{_xA_uB_yA_v B_
zA_wB_t}\ar[r]^-{\sa{III\widetilde{R}I}}
\ar[d]|<(.25){\sa{I\widetilde{R}III}} &\sv{_xA_uB_{{y}}A_v A_{q} B_w
B_t}\ar[r]^-{\sa{IIII b}}
\ar@{}[dr]|*+[o][F-]{F}\ar[d]|<(.25){\sa{II\widetilde{R}II\circ
I\widetilde{R}III}} &\sv{_xA_uB_yA_v A_q B_t}
\ar@{}[dr]|*+[o][F-]{R}\ar[r]^-{\sa{IIaI}}
\ar[d]|-{\sa{II\widetilde{R}I\circ I\widetilde{R}II}} &\sv{_xA_uB_y A_q B_t}
\ar[d]|-{\sa{I\widetilde{R}I}} \\ \sv{_xA_uA_pB_vB_zA_wB_t}
\ar@/^10px/[r]^-{\sa{II\widetilde{R}II\circ III\widetilde{R}I}}
\ar[d]|-{\sa{IIaII}}\ar@{}[dr]|*+[o][F-]{F} &\sv{_xA_u A_p A_{|pvq|} B_q B_w
B_t}\ar@{}[dr]|*+[o][F-]{F} \ar[r]_-{\sa{IIIIb}} \ar[d]|-{aIIII} &\sv{_xA_u
A_p A_{|pvq|} B_q B_t}\ar[r]_-{\sa{IaII}}
\sv{\ar[d]|-{aIII}}\ar@{}[dr]|*+[o][F-]{A} &\sv{_x A_u A_{|pvq|} B_q
B_t}\ar[d]|-{\sa{aII}} \\ \sv{_xA_pB_vB_zA_wB_t}
\ar@{}[dr]_>(.60)*+[o][F-]{L} \ar@/^-10px/[r]_-{\sa{I\widetilde{R}II\circ
II\widetilde{R}I}} \ar[d]|-{\sa{IbII}} & \sv{_x A_p A_{|pvq|} B_q B_w B_t}
\ar[r]_-{\sa{IIIb}}\sv{\ar[d]|-{IIbI}}\ar@{}[dr]|*+[o][F-]{A} &\sv{_x A_p
A_{|pvq|} B_q B_t}\sv{\ar[r]_-{\sa{aII}}}\sv{\ar[d]|-{IIb}}
\ar@{}[dr]|*+[o][F-]{F}&\sv{_x A_{|pvq|} B_q B_t} \ar[d]|-{\sa{Ib}} \\
\sv{_xA_p B_z A_w B_t} \ar[r]_-{\sa{I\widetilde{R}I}} &\sv{_xA_p A_{|pvq|}
B_w B_t} \sv{\ar[r]_{IIb}} &\sv{_xA_p A_{|pvq|} B_t}\sv{\ar[r]_{aI}}
&\sv{_xA_{|pvq|} B_t} }
\end{equation*}%
\end{enumerate}
\end{lemma}

\begin{proof}
Since $\boldsymbol{M}$ is a domain it follows that any subfactor of $%
_{x}A_{u}B_{y}A_{v}B_{z}A_{w}B_{t}$ is not an initial object. In particular $%
_{v}B_{z}A_{w}\neq \emptyset .$ Thus, by the definition of simple twisting
systems, $_{v}A_{q}B_{w}$ is not an initial object. In conclusion, $%
_{v}A_{q} $ and $_{q}B_{w}$ are not initial objects in $\boldsymbol{M}.$
Since $_{u}B_{y}A_{v}\neq \emptyset $ it follows that $_{u}B_{y}A_{v}A_{q}%
\neq \emptyset $. In view of the definition of simple twisting systems (the
second condition) we deduce that $\left\vert pvq\right\vert =\left\vert
uyq\right\vert $. The other relation can be proved in a similar way.

Let $f$ and $g$ denote the following two morphisms
\begin{equation*}
f:={}_{x}A_{u}a_{\left\vert pvq\right\vert }^{p}B_{q}B_{t}\circ
{}_{x}A_{u}A_{p}\widetilde{R}_{q}^{v}B_{t}\circ {}_{x}A_{u}\widetilde{R}%
_{v}^{y}A_{q}B_{t}\text{\quad and\quad }g:={}_{x}A_{u}\widetilde{R}%
_{q}^{y}B_{t}\circ {}_{x}A_{u}B_{y}a_{q}^{v}B_{t}.
\end{equation*}%
The target of $f$ is $_{x}A_{u}A_{\left\vert pvq\right\vert }B_{q}B_{t},$
while the codomain of $g$ is $_{x}A_{u}A_{\left\vert uyq\right\vert
}B_{q}B_{t}$. These two objects may be different for some elements $x,u,t,p$
and $q$ in $S$. Thus, in general, it does not make sense to speak about the
square (R). On the other hand, we have seen that $\left\vert uyq\right\vert
=\left\vert pvq\right\vert ,$ if $p=\left\vert uyv\right\vert $ and $%
q=\left\vert vzw\right\vert .$ Hence (R) is well defined for these values of
$p$ and $q$. Furthermore, since $_{u}B_{y}A_{v}A_{q}\neq \emptyset ,$ by
definition of simple twisting systems we have
\begin{equation}
_{u}a_{|pvq|}^{p}B_{q}\circ {}_{u}A_{p}\widetilde{R}_{q}^{v}\circ {}_{u}%
\widetilde{R}_{v}^{y}A_{q}={}_{u}\widetilde{R}_{q}^{y}\circ
{}_{u}B_{y}a_{q}^{v}.  \label{lem}
\end{equation}%
By tensoring both sides of the above relation with $_{x}A_{u}$ on the left
and with $_{q}B_{t}$ on the right we get that $f=g,$ i.e. (R) is
commutative. Analogously, one shows that (L) is well defined and
commutative. All other squares are well defined by construction, their
arrows targeting to the right objects. The squares (F) are commutative since
the tensor product is a functor. The remaining squares (A) are commutative
by associativity.
\end{proof}

\begin{theorem}
\label{thm:TTP} Let $\boldsymbol{M}$ be an $S$-distributive monoidal domain.
If $(\widetilde{{R}},\left\vert \dots \right\vert )$ is a simple twisting
system, then the data in \S \ref{fa:ARB} define an $\boldsymbol{M}$-category
$\boldsymbol{A}\otimes _{{R}}\boldsymbol{B}$ that factorizes through $%
\boldsymbol{A}$ and $\boldsymbol{B}$.
\end{theorem}

\begin{proof}
Let us assume that $_{x}A_{u}B_{y}A_{v}B_{z}A_{w}B_{t}\neq \emptyset $. In
view of the previous lemma, the outer square in the diagram from Lemma \ref{le:STS} (2) is commutative. It follows that
\begin{equation*}
_{x}c_{t}^{y}\circ {}_{x}A_{\overline{u}}B_{y}c_{t}^{z}\circ {}_{x}\sigma
_{y}^{u}\sigma _{z}^{v}\sigma _{t}^{w}={}_{x}c_{t}^{z}\circ
{}_{x}c_{z}^{y}A_{\overline{w}}B_{t}\circ {}_{x}\sigma _{y}^{u}\sigma
_{z}^{v}\sigma _{t}^{w}\,{}.
\end{equation*}%
If $_{x}A_{u}B_{y}A_{v}B_{z}A_{w}B_{t}=\emptyset $ this identity obviously
holds. Since $_{x}A_{\overline{u}}B_{y}A_{\overline{v}}B_{z}A_{\overline{w}%
}B_{t}$ is the coproduct of $\{_{x}A_{u}B_{y}A_{v}B_{z}A_{w}B_{t}\}_{u,v,w%
\in S}$, with the canonical inclusions $\{_{x}\sigma _{y}^{u}\sigma
_{z}^{v}\sigma _{t}^{w}\}_{u,v,w\in S},$ we deduce that the composition in $%
\boldsymbol{A}\otimes _{{R}}\boldsymbol{B}$ is associative.

We apply the same strategy to show that $1_{x}:={}_{x}\sigma _{x}^{x}\circ
(1_{x}^{A}\otimes 1_{x}^{B})$ is a left identity map of $x,$ that is we have
$_{x}c_{y}^{x}\circ (1_{x}\otimes {}_{x}A_{\overline{u}}B_{y})={}_{x}A_{%
\overline{u}}B_{y}$ for any $y.$ By the universal property of coproducts and
the definition of the composition in $\boldsymbol{A}\otimes _{{R}}%
\boldsymbol{B},$ it is enough to prove that
\begin{equation}
_{x}\sigma _{y}^{\left\vert xxu\right\vert }\circ {}_{x}a_{\left\vert
xxu\right\vert }^{x}b_{y}^{u}\circ {}_{x}A_{x}\widetilde{R}%
_{u}^{x}B_{y}\circ (1_{x}^{A}\otimes 1_{x}^{B}\otimes
{}_{x}A_{u}B_{y})={}_{x}\sigma _{y}^{u},  \label{ec:unit}
\end{equation}%
for all $u\in S.$ If $_{x}A_{u}$ is an initial object we have nothing to
prove, as the domains of the sides of the above equation are also initial
objects (recall that $_{x}A_{u}B_{y}=\emptyset $ if $_{x}A_{u}=\emptyset $).
Let us suppose that $_{x}A_{u}$ is not an initial object. Then by the
definition of simple twisting systems (the third condition) we get $%
\left\vert xxu\right\vert =u$ and%
\begin{equation*}
_{x}\sigma _{y}^{\left\vert xxu\right\vert }\circ {}_{x}a_{\left\vert
xxu\right\vert }^{x}b_{y}^{u}\circ {}_{x}A_{x}\widetilde{R}%
_{u}^{x}B_{y}\circ (1_{x}^{A}\otimes 1_{x}^{B}\otimes
{}_{x}A_{u}B_{y})={}_{x}\sigma _{y}^{u}\circ {}_{x}a_{u}^{x}b_{y}^{u}\circ
(1_{x}^{A}\otimes {}_{x}A_{u}\otimes 1_{u}^{B}\otimes {}_{u}B_{y}).
\end{equation*}%
Thus the equation (\ref{ec:unit}) immediately follows by the fact $1_{x}^{A}$
and $1_{u}^{B}$ are the identity morphisms of $x$ and $u$. The fact that $%
1_{x}$ is a right identity map of $x$ can be proved analogously.

We now claim that $\{_{x}\alpha _{y}\}_{x,y\in S}$ is an $\boldsymbol{M}$%
-functor. Taking into account the definition of $\boldsymbol{\alpha }$ and $%
_{x}c_{z}^{y}$ we must prove that
\begin{equation}
_{x}\sigma _{z}^{\left\vert yyz\right\vert }\circ {}_{x}a_{\left\vert
yyz\right\vert }^{x}a_{z}^{z}\circ {}_{x}A_{y}\widetilde{R}%
_{z}^{y}A_{z}\circ (_{x}A_{y}\otimes 1_{y}^{A}\otimes {}_{y}A_{z}\otimes
1_{z}^{A}{})={}_{x}\sigma _{z}^{z}\circ ({}_{x}a_{z}^{y}\otimes 1_{z}^{A}),
\label{ec:alpha}
\end{equation}%
for all $x,$ $y$ and $z$ in $S.$ Once again, if $_{y}A_{z}=\emptyset $ we
have nothing to prove. In the other case, one can proceed as in the proof of
(\ref{ec:unit}) to get this equation. Similarly, $\boldsymbol{\beta }$ is an
$\boldsymbol{M}$-functor.

It remains to prove the fact that $\boldsymbol{A}\otimes _{{R}}\boldsymbol{B}
$ factorizes through $\boldsymbol{A}$ and $\boldsymbol{B}.$ As a matter of
fact, for this enriched category, we shall show that $_{x}\varphi _{y}$ is
the identity map of $_{x}(\boldsymbol{A}\otimes _{{R}}\boldsymbol{B})_{y},$
for all $x$ and $y$ in $S.$ Recall that $_{x}\varphi _{y}$ is the unique map
such that $_{x}\varphi _{y}\circ {}_{x}\sigma _{y}^{u}={}_{x}c_{y}^{u}\circ
{}_{x}\alpha _{u}\beta _{y},$ for all $u\in S.$ Hence to conclude the proof
of the theorem it is enough to obtain the following relation%
\begin{equation}
{}_{x}\sigma _{y}^{\left\vert uuu\right\vert }\circ {}_{x}a_{\left\vert
uuu\right\vert }^{u}b_{y}^{u}\circ {}_{x}A_{u}\widetilde{R}%
_{u}^{u}B_{y}\circ (_{x}A_{u}\otimes 1_{u}^{B}\otimes 1_{u}^{A}\otimes
{}{}_{u}B_{z}{})={}_{x}\sigma _{y}^{u},  \label{ec:fact}
\end{equation}%
for all $u\in S.$ We may suppose that $_{x}A_{u}$ is not initial object.
Thus $\left\vert uuu\right\vert =u$ and we can take $x=u$ and $y=u$ in (\ref%
{cor1-4}). Hence, using the same reasoning as in the proof of (\ref{ec:unit}%
), we deduce the required identity.
\end{proof}

\begin{corollary}
\label{cor: TTP} Let $\boldsymbol{A}$ and $\boldsymbol{B}$ be enriched
categories over an $S$-distributive monoidal category $\boldsymbol{M}$. Let
us suppose that for all $x,$ $y,$ $z$ and $t$ in $S$ the function $%
\left\vert \cdots \right\vert :S^{3}\rightarrow S$ satisfies the equations%
\begin{equation}
|xy|yzt||=|xzt|,\quad \quad ||xyz|zt|=|xyt|,\quad \quad |xxy|=y\quad \quad
\text{and}\quad \quad |xyy|=x.  \label{modul}
\end{equation}%
If $\{_{x}\widetilde{{R}}_{z}^{y}\}_{x,y,z\in S}$ is a family of maps which
satisfies the identities (\ref{cor1-1})-(\ref{cor1-4}) for all $x,$ $y,$ $z$
and $t$ in $S,$ then the data in \S \ref{fa:ARB} define an $\boldsymbol{M}$%
-category $\boldsymbol{A}\otimes _{{R}}\boldsymbol{B}$ that factorizes
through $\boldsymbol{A}$ and $\boldsymbol{B}$.
\end{corollary}

\begin{proof}
Let $x,$ $y,$ $u,$ $v,$ $z,$ $w$ and $t$ be arbitrary elements in $S.$ By
using the first two identities in (\ref{modul}) we get $\left\vert
uyq\right\vert =\left\vert pvq\right\vert =\left\vert pzw\right\vert ,$
where $p=\left\vert uyv\right\vert $ and $q=\left\vert vzw\right\vert .$
Hence the first statement in Lemma \ref{le:STS} is true. In particular, the
squares (R) and (L) in the diagram from Lemma \ref{le:STS}(2) are well
defined. On the other hand, under the assumptions of the corollary, the
relation (\ref{lem}) hold. Therefore we can continue as in the proof of the
second part of Lemma \ref{le:STS} to show that (R) is commutative.
Similarly, (L) is commutative too. It follows that the outer square of is
commutative too. By the universal property of the coproduct we deduce that
the composition is associative, see the first paragraph of the proof of
Theorem \ref{thm:TTP}.

Furthermore, the relations in (\ref{modul}) together with the identities (%
\ref{cor1-1})-(\ref{cor1-4}) imply the equations (\ref{ec:unit}), (\ref%
{ec:alpha}) and (\ref{ec:fact}). Proceeding as in the proof of Theorem \ref%
{thm:TTP} we conclude that $\boldsymbol{A}\otimes _{{R}}\boldsymbol{B}$ is
an $\boldsymbol{M}$-category that factorizes through $\boldsymbol{A}$ and $%
\boldsymbol{B}$.
\end{proof}

\begin{remark}
Throughout this remark we assume that $\boldsymbol{M}$ is a $T$-distributive
monoidal category, where $T$ is an arbitrary set. In other words, any family
of objects in $\boldsymbol{M}$ has a coproduct and the tensor product is
distributive over all coproducts. It was noticed in \cite[\S 2.1 and \S 2.2]%
{RW} that, for such a monoidal category $\boldsymbol{M},$ one can define a
bicategory $\boldsymbol{M}$-$\boldsymbol{mat}$ as follows.

By construction, its $0$-cells are arbitrary sets. If $I$ and $J$ are two
sets, then the $1$-cells in $\boldsymbol{M}$-$\boldsymbol{mat}$ from $I$ to $%
J$ are the $J\times I$-indexed families of objects in $\boldsymbol{M}$. A $2$%
-cell with source $\{X_{ji}\}_{(j,i)\in J\times I}$ and target $%
\{Y_{ji}\}_{(j,i)\in J\times I}$ is a family $\{f_{ji}\}_{(j,i)\in J\times
I} $ of morphisms $f_{ji}:X_{ji}\rightarrow Y_{ji}.$ The composition of the $%
1$-cells $\{X_{kj}\}_{(k,j)\in K\times J}$ and $\{Y_{ji}\}_{(j,i)\in J\times
I} $ is the family $\{Z_{ki}\}_{(k,i)\in K\times I},$ where
\begin{equation*}
Z_{ki}:=\textstyle\bigoplus_{j\in J}X_{kj}\otimes Y_{ji}.
\end{equation*}%
The vertical composition in $\boldsymbol{M}$-$\boldsymbol{mat}$ of $%
\{f_{ji}\}_{(j,i)\in J\times I}$ and $\{g_{ji}\}_{(j,i)\in J\times I}$ makes
sense if and only if the source of $f_{j i}$ and the target of $g_{j i}$ are
equal for all $i$ and $j.$ If it exists, then it is defined by%
\begin{equation*}
\{f_{ji}\}_{(j,i)\in J\times I}\bullet \{g_{ji}\}_{(j,i)\in J\times
I}=\{f_{ji}\circ g_{ji}\}_{(j,i)\in J\times I}.
\end{equation*}
Let $\{f_{ji}\}_{(j,i)\in J\times I}$ and $\{f_{kj}^{\prime }\}_{(k,j)\in
K\times J}$ be $2$-cells such that $f_{ji}:X_{ji}\rightarrow Y_{ji}$ and $%
f_{kj}^{\prime }:X_{kj}^{\prime }\rightarrow Y_{kj}^{\prime }.$ By the
universal property of coproducts, for each $(k,i)\in K\times I$, there
exists a unique morphism $h_{ki}:\bigoplus_{j\in J}X_{kj}^{\prime }\otimes
X_{ji}\rightarrow \bigoplus_{j\in J}Y_{kj}^{\prime }\otimes Y_{ji}$ whose
restriction to $X_{kj}^{\prime }\otimes $ $X_{ji}$ is $f_{kj}^{\prime
}\otimes $ $f_{ji}$. By definition, the horizontal composition of $%
\{f_{kj}^{\prime }\}_{(k,j)\in K\times J}$ and $\{f_{ji}\}_{(j,i)\in J\times
I}$ is the family $\{h_{ki}\}_{(k,i)\in K\times I}.$ The identity $1$-cells
and $2$-cells in $\boldsymbol{M}$-$\boldsymbol{mat}$ are the obvious ones.

As pointed out in \cite{RW}, a monad on a set $S$ in $\boldsymbol{M}$-$%
\boldsymbol{mat}$ is an $\boldsymbol{M}$-category with the set of objects $S$%
, and conversely. In particular, given two $\boldsymbol{M}$-categories with
the same set of objects, one may speak about distributive laws between the
corresponding monads in $\boldsymbol{M}$-$\boldsymbol{mat}.$ In our
terminology, they are precisely the twisting systems. In view of \cite[\S 3.1%
]{RW}, factorizable enriched categories generalize strict factorization
systems.

In conclusion, the Theorem \ref{teo1} may be regarded as a version of \cite[%
Proposition 3.3]{RW} for enriched categories. For a simple twisting system $(%
\widetilde{R},\left\vert \cdots \right\vert )$ between $\boldsymbol{B}$ and $%
\boldsymbol{A},$ the enriched category $\boldsymbol{A}\otimes _{R}%
\boldsymbol{B}$ that we constructed in Theorem \ref{thm:TTP} can also be
described in terms of monads. Let $\rho :B\circ A\rightarrow A\circ B$
denote the distributive law associated to $(\widetilde{R},\left\vert \cdots
\right\vert ),$ where $(A,m_{A},1_{A})$ and $(B,m_{B},1_{B})$ are the monads
in $\boldsymbol{M}$-$\boldsymbol{mat}$ corresponding to $\boldsymbol{A}$ and
$\boldsymbol{B,}$ respectively. By the general theory of monads in a
bicategory, it follows that $A\circ B$ is a monad in $\boldsymbol{M}$-$%
\boldsymbol{mat}$ with respect to the multiplication and the unit given by
the formulae:%
\begin{equation*}
m:=\left( m_{A}\circ m_{B}\right) \bullet \left( {\text{Id}}_{A}\circ \rho
\circ {\text{Id}}_{B}\right) \text{\quad and\quad }1:=1_{A}\circ 1_{B}.
\end{equation*}
It is not difficult to show that $\boldsymbol{A}\otimes _{R}\boldsymbol{B}$
is the $\boldsymbol{M}$-category associated to $(A\circ B,m,1).$

By replacing $\boldsymbol{Set}$-$\boldsymbol{mat}$ with a suitable
bicategory, one obtains similar results for other algebraic structures, such
as PROs and PROPs; see \cite{La}. We also would like to note that
distributive laws between pseudomonads are investigated in \cite{Mar}.

We are indebted to the referee for pointing the papers \cite{La, Mar, RW}
out to us.
\end{remark}

\section{Matched pairs of enriched categories.}

Throughout this section $(\boldsymbol{M^{\prime }},\otimes ,\boldsymbol{1}%
,\chi )$ denote a braided category and we take $\boldsymbol{M}$ to be the
monoidal category $\boldsymbol{Coalg({M^{\prime }})}.$ Our aim is to
characterize simple twisting systems between two categories that are
enriched over $\boldsymbol{M}$. We start by investigating some properties of
the morphisms in $\boldsymbol{M}$. For the moment, we impose no conditions
on $\boldsymbol{M^{\prime }}$.

A slightly more general version of the following lemma is stated in \cite[%
Proposition 3.2]{La}. For the sake of completeness we include a proof of it.

\begin{lemma}
\label{le:coalgebra}Let $(C,\Delta _{C},\varepsilon _{C})$, $(D_{1},\Delta
_{D_{1}},\varepsilon _{D_{1}})$ and $(D_{2},\Delta _{D_{2}},\varepsilon
_{D_{2}})$ be coalgebras in $\boldsymbol{M^{\prime }}$. Let $f:C\rightarrow
D_{1}\otimes D_{2}$ be a morphism of coalgebras. Then $f_{1}:=(D_{1}\otimes
\varepsilon _{D_{2}})\circ f$ and $f_{2}:=(\varepsilon _{D_{1}}\otimes
D_{2})\circ f$ are coalgebra morphisms and the following relations hold:
\begin{align}
(f_{1}\otimes f_{2})\circ \Delta _{C}& =f,  \label{f1} \\
(f_{2}\otimes f_{1})\circ \Delta _{C}& =\chi _{D_{1},D_{2}}\circ
(f_{1}\otimes f_{2})\circ \Delta _{C}.  \label{f2}
\end{align}%
Conversely, let $f_{1}:C\rightarrow D_{1}$ and $f_{2}:C\rightarrow D_{2}$ be
coalgebra morphisms such that \eqref{f2} holds. Then $f:=(f_{1}\otimes
f_{2})\circ \Delta _{C}$ is a coalgebra map such that
\begin{equation}
(D_{1}\otimes \varepsilon _{D_{2}})\circ f=f_{1}\quad \text{\emph{and}}\quad
(\varepsilon _{D_{1}}\otimes D_{2})\circ f=f_{2}.  \label{le:f}
\end{equation}
\end{lemma}

\begin{proof}
Let us assume that $f:C\rightarrow D_{1}\otimes D_{2}$ is a coalgebra
morphism. Let $\varepsilon _{i}:=\varepsilon _{D_{i}}$, for $i=1,2.$
Clearly, $D_{1}\otimes \varepsilon _{D_{2}}$ and $\varepsilon
_{D_{1}}\otimes D_{2}$ are coalgebra morphisms. In conclusion $f_{1}$ and $%
f_{2}$ are morphisms in $\boldsymbol{M}$. On the other hand, as $f$ is a
morphism in $\boldsymbol{M}$ we have%
\begin{equation}
(D_{1}\otimes \chi _{D_{1},D_{2}}\otimes D_{2})\circ (\Delta _{D_{1}}\otimes
\Delta _{D_{2}})\circ f=(f\otimes f)\circ \Delta _{C}.  \label{f3}
\end{equation}%
Hence, using the definition of $f_{1}$ and $f_{2}$, the relation \eqref{f3},
the fact that the braiding is a natural transformation and the compatibility
relation between the comultiplication and the counit we get%
\begin{align*}
(f_{1}\otimes f_{2})\circ \Delta _{C}& =(D_{1}\otimes \varepsilon
_{2}\otimes \varepsilon _{1}\otimes D_{2})\circ (f\otimes f)\circ \Delta _{C}
\\
& =\left[D_{1}\otimes \left((\varepsilon _{2}\otimes \varepsilon _{1})\circ
\chi _{D_{1,}D_{2}}\otimes D_{2}\right)\circ (\Delta _{D_{1}}\otimes \Delta
_{D_{2}})\right]\circ f \\
& =(D_{1}\otimes \varepsilon _{1}\otimes \varepsilon _{2}\otimes D_{2})\circ
(\Delta _{D_{1}}\otimes \Delta _{D_{2}})\circ f=f.
\end{align*}%
By applying $\varepsilon _{1}\otimes D_{2}\otimes D_{1}\otimes \varepsilon
_{2}$ to (\ref{f3}) and using once again the compatibility between the
comultiplication and the counit we obtain%
\begin{align*}
(f_{2}\otimes f_{1})\circ \Delta _{C}& =(\varepsilon _{1}\otimes
D_{2}\otimes D_{1}\otimes \varepsilon _{2})\circ (f\otimes f)\circ \Delta
_{C} \\
& =(\varepsilon _{1}\otimes D_{2}\otimes D_{1}\otimes \varepsilon _{2})\circ
(D_{1}\otimes \chi _{D_{1},D_{2}}\otimes D_{2})\circ (\Delta _{D_{1}}\otimes
\Delta _{D_{2}})\circ f \\
& =\chi _{D_{1},D_{2}}\circ (\varepsilon _{1}\otimes D_{1}\otimes
D_{2}\otimes \varepsilon _{2})\circ (\Delta _{D_{1}}\otimes \Delta
_{D_{2}})\circ f =\chi _{D_{1},D_{2}}\circ f.
\end{align*}%
Conversely, let us assume that $f_{1}:C\rightarrow D_{1}$ and $%
f_{2}:C\rightarrow D_{2}$ are morphisms in $\boldsymbol{M}$ such that %
\eqref{f2} holds. Let $f:=(f_{1}\otimes f_{2})\circ $ $\Delta _{C}.$ By the
definition of the comultiplication on $D_{1}\otimes D_{2}$ and the fact that
$f_{1}$ and $f_{2}$ are morphisms in $\boldsymbol{M,}$ we get%
\begin{align*}
\Delta _{D_{1}\otimes D_{2}}\circ f&= (D_{1}\otimes \chi
_{D_{1},D_{2}}\otimes D_{2})\circ (\Delta _{D_{1}}\otimes \Delta
_{D_{2}})\circ (f_{1}\otimes f_{2})\circ \Delta _{C} \\
&= (D_{1}\otimes \chi _{D_{1},D_{2}}\otimes D_{2})\circ (f_{1}\otimes
f_{1}\otimes f_{2}\otimes f_{2})\circ (\Delta _{C}\otimes \Delta _{C})\circ
\Delta _{C}.
\end{align*}%
Taking into account (\ref{f2}) and the fact that the comultiplication is
coassociative, it follows that%
\begin{align*}
\Delta _{D_{1}\otimes D_{2}}\circ f &= \left[f_{1}\otimes \left( \chi
_{D_{1},D_{2}}\circ (f_{1}\otimes f_{2})\circ \Delta _{C}\right)\otimes f_{2}%
\right] \circ (C\otimes \Delta _{C})\circ \Delta _{C} \\
&=\left[ f_{1}\otimes \left( (f_{2}\otimes f_{1})\circ \Delta _{C}\right)
\otimes f_{2}\right] \circ (C\otimes \Delta _{C})\circ \Delta _{C} \\
&=\left[\left((f_{1}\otimes f_{2})\circ \Delta _{C}\right)\otimes
\left((f_{1}\otimes f_{2}\right)\circ \Delta _{C})\right]\circ \Delta
_{C}=(f\otimes f)\circ \Delta _{C}.
\end{align*}%
The formula that defines $f$ together with $\varepsilon _{i}\circ
f_{i}=\varepsilon _{C}$ yield
\begin{equation*}
(\varepsilon _{1}\otimes \varepsilon _{2})\circ f=(\varepsilon _{1}\circ
f_{1}\otimes \varepsilon _{2}\circ f_{2})\circ \Delta _{C}=(\varepsilon
_{C}\otimes \varepsilon _{C})\circ \Delta _{C}=\varepsilon _{C}.
\end{equation*}%
Thus $f$ is a morphism of coalgebras, so the lemma is proved. The equations
in \eqref{le:f} are obvious, as $\varepsilon _{i}\circ f_{i}=\varepsilon
_{C} $.
\end{proof}

\begin{remark}
\label{re:f=g}Let $f^{\prime },f^{\prime \prime }:C\rightarrow D_{1}\otimes
D_{2}$ be coalgebra morphisms. By the preceding lemma, $f^{\prime }$ and $%
f^{\prime \prime }$ are equal if and only if
\begin{equation*}
(\varepsilon _{1}\otimes D_{2})\circ f^{\prime }=(\varepsilon _{1}\otimes
D_{2})\circ f^{\prime \prime }\text{\quad and\quad }(D_{1}\otimes
\varepsilon _{2})\circ f^{\prime }=(D_{1}\otimes \varepsilon _{2})\circ
f^{\prime \prime }.\text{ }
\end{equation*}
\end{remark}

\begin{fact}[The morphisms $\boldsymbol{_{x}\triangleright _{z}^{y}}$ and $%
\boldsymbol{_{x}\triangleleft _{z}^{y}}$.]
\label{coalg(M')} Let $\boldsymbol{A}$ and $\boldsymbol{B}$ denote two $%
\boldsymbol{M}$-categories whose objects are the elements of a set $S$. The
hom-objects of $\boldsymbol{A}$ and $\boldsymbol{B}$ are coalgebras, which
will be denoted by $(_{x}A_{y},{}_{x}\Delta _{y}^{A},{}_{x}\varepsilon
_{y}^{A})$ and $(_{x}B_{y},{}_{x}\Delta _{y}^{B},{}_{x}\varepsilon _{y}^{B})$%
. By definition, the composition and the identity maps in $\boldsymbol{A}$
and $\boldsymbol{B}$ are coalgebra morphisms. Note that the comultiplication
of $_{x}B_{y}A_{z}$ is given by
\begin{equation*}
\Delta _{_{x}B_{y}A_{z}}=(_{x}B_{y}\otimes \chi
_{_{x}B_{y},_{y}A_{z}}\otimes {}_{y}A_{z})\circ {}_{x}\Delta _{y}^{B}\Delta
_{z}^{A}.
\end{equation*}%
Let $|\cdots |:S^{3}\rightarrow S$ be a function and let $\widetilde{R}$
denote an $S^{3}$-indexed family of coalgebra morphisms $_{x}\widetilde{R}%
_{z}^{y}:{}_{x}B_{y}A_{z}\rightarrow {}_{x}A_{|xyz|}B_{z}$. We define $%
_{x}\triangleright _{z}^{y}:{}_{x}B_{y}A_{z}\rightarrow {}_{x}A_{|xyz|}$ and
$_{x}\triangleleft _{z}^{y}:{}_{x}B_{y}A_{z}\rightarrow {}_{|xyz|}B_{z}$ by
\begin{equation}
_{x}\triangleright _{z}^{y}:={}_{x}A_{\left\vert xyz\right\vert }\varepsilon
_{z}^{B}\circ {}_{x}\widetilde{R}_{z}^{y}\quad \text{and}\quad
_{x}\triangleleft _{z}^{y}:=({}_{x}\varepsilon _{\left\vert xyz\right\vert
}^{A}B_{z})\circ {}_{x}\widetilde{R}_{z}^{y}.  \label{lr}
\end{equation}%
In view of Lemma \ref{le:coalgebra}, $_{x}\triangleright _{z}^{y}$ and $%
_{x}\triangleleft _{z}^{y}$ are coalgebra morphisms and they satisfy the
relations%
\begin{align}
(_{x}\triangleright _{z}^{y}\otimes {}_{x}\triangleleft _{z}^{y})\circ
\Delta _{_{x}B_{y}A_{z}}& ={}_{x}\widetilde{R}_{z}^{y},  \label{mp1} \\
(_{x}\triangleleft _{z}^{y}\otimes {}_{x}\triangleright _{z}^{y})\circ
\Delta _{{_{x}B_{y}A_{z}}}& =\chi _{_{x}A_{|xyz|},_{|xyz|}B_{z}}\circ
(_{x}\triangleright _{z}^{y}\otimes {}_{x}\triangleleft _{z}^{y})\circ
\Delta _{_{x}B_{y}A_{z}}.  \label{mp2}
\end{align}%
Conversely, if one starts with $\triangleright :=\{_{x}\triangleright
_{z}^{y}\}_{x,y,z\in S}$ and ${\triangleleft :=}$ $\{_{x}{\triangleleft }%
_{z}^{y}\}_{x,y,z\in S},$ two families of coalgebra maps that satisfy (\ref%
{mp2}), then by formula (\ref{mp1}) we get the set $\widetilde{R}:={}\{_{x}%
\widetilde{R}_{z}^{y}\}_{x,y,z\in S}$ whose elements are coalgebra maps, cf.
Lemma \ref{le:coalgebra}. Therefore, there is an one-to-one correspondence
between the couples $(\triangleright ,{\triangleleft )}$ and the sets $%
\widetilde{R}$ as above. Our goal is to characterize those couples $%
(\triangleright ,{\triangleleft )}$ that corresponds to a simple twisting
system in $\boldsymbol{{M^{\prime }}}.$
\end{fact}

\begin{lemma}
\label{le:lr}The statements below are true.

\begin{enumerate}
\item If $\left\vert xy\left\vert yzt\right\vert \right\vert =\left\vert
xzt\right\vert $ then the relation (\ref{cor1-1}) is equivalent to the
following equations:
\begin{align}
{_{x}\triangleleft }_{t}^{z}\circ {}_{x}b_{z}^{y}A_{t}& ={}_{\left\vert
xzt\right\vert }b_{t}^{\left\vert yzt\right\vert }\circ {_{x}\triangleleft }%
_{\left\vert yzt\right\vert }^{y}B_{t}\circ (_{x}B_{y}\otimes {%
_{y}\triangleright }_{t}^{z}\otimes {_{y}\triangleleft }_{t}^{z})\circ
(_{x}B_{y}\otimes \Delta _{{}_{y}B_{z}A_{t}}),  \label{lr1-1} \\
{_{x}\triangleright }_{t}^{z}\circ {}_{x}b_{z}^{y}A_{t}& ={_x\triangleright}%
_{|yzt|}^{y}\circ {}_{x}B{_y\triangleright} _{t}^{z}.  \label{lr1-2}
\end{align}

\item If $\left\vert xyz\left\vert zt\right\vert \right\vert =\left\vert
xyt\right\vert $ then the relation (\ref{cor1-2}) is equivalent to the
following equations:
\begin{align}
{_{x}\triangleright }_{t}^{y}\circ {}_{x}B_{y}a_{t}^{z}&
={}_{x}a_{\left\vert xyt\right\vert }^{\left\vert xyz\right\vert }\circ
{}_{x}A_{\left\vert xyz\right\vert }\triangleright _{t}^{z}\circ \;({%
_x\triangleright} _{z}^{y}\otimes {_{x}\triangleleft} _{z}^{y}\otimes
{}_{z}A_{t})\circ (\Delta _{{_{x}B_{y}A_{z}}}\otimes {}_{z}A_{t}),
\label{lr2-1} \\
{_{x}\triangleleft }_{t}^{y}\circ {}_{x}B_{y}a_{t}^{z}& ={_{\left\vert
xyz\right\vert }\triangleleft }_{t}^{z}\circ \;{_{x}\triangleleft }%
_{z}^{y}A_{t}.  \label{lr2-2}
\end{align}

\item If $\left\vert xyy\right\vert =x$ then the relation (\ref{cor1-3}) is
equivalent to the following equations:
\begin{align}
{_{x}\triangleleft }_{y}^{x}\circ \left( 1_{x}^{B}\otimes {}_{x}A_{y}\right)
& ={}_{x}{}\varepsilon _{y}^{A}\otimes 1_{y}^{B},  \label{lr3-1} \\
{_{x}\triangleright }_{y}^{x}\circ \left( 1_{x}^{B}\otimes
{}_{x}A_{y}\right) & ={_{x}A}_{y}{}.  \label{lr3-2}
\end{align}

\item If $\left\vert xxy\right\vert =y$ then the relation (\ref{cor1-4}) is
equivalent to the following equations:%
\begin{align}
{_{x}\triangleleft }_{y}^{y}\circ \left( _{x}B_{y}\otimes 1_{y}^{A}{}\right)
& ={_{x}B}_{y},  \label{lr4-1} \\
{_{x}\triangleright }_{y}^{y}\circ \left( _{x}B_{y}\otimes 1_{y}^{A}\right)
& =1_{x}^{A}\otimes {}_{x}\varepsilon _{y}^{B}{}.  \label{lr4-2}
\end{align}
\end{enumerate}
\end{lemma}

\begin{proof}
In order to prove the first statement we apply the Remark \ref{re:f=g} to
\begin{equation*}
f^{\prime }:={}_{x}\widetilde{R}_{t}^{z}\circ {}_{x}b_{z}^{y}A_{t}\quad
\text{and}\quad f^{\prime \prime }:={}_{x}A_{|xzt|}b_{t}^{|yzt|}\circ {}_{x}%
\widetilde{R}_{|yzt|}^{y}B_{t}\circ {}_{x}B_{y}\widetilde{R}_{t}^{z}.
\end{equation*}%
Note that $f^{\prime \prime }$ is well defined and its target is $%
_{x}A_{|xzt|}B_{t},$ since the codomain of ${}_{x}\widetilde{R}%
_{|yzt|}^{y}B_{t}\circ {}_{x}B_{y}\widetilde{R}_{t}^{z}$ is $%
_{x}A_{\left\vert xy\left\vert yzt\right\vert \right\vert }B_{t}$ and $%
\left\vert xy\left\vert yzt\right\vert \right\vert =\left\vert
xzt\right\vert $. Clearly, $f^{\prime }$ and $f^{\prime \prime }$ are
coalgebra morphisms, since the composite and the tensor product of two
morphisms in $\boldsymbol{M}$ remain in $\boldsymbol{M}.$ An easy
computation, based on the equation (\ref{mp1}) and the formulae of $%
_{x}\triangleright _{z}^{y}$ and $_{x}\triangleleft _{z}^{y}$, yields us%
\begin{align*}
_{x}\varepsilon _{|xzt|}^{A}B_{t}\circ f^{\prime }& ={}_{x}\triangleleft
_{t}^{z}\circ {}_{x}b_{z}^{y}A_{t}, \\
_{x}A_{|xzt|}\varepsilon _{t}^{B}\circ f^{\prime }& ={}_{x}\triangleright
_{t}^{z}\circ {}_{x}b_{z}^{y}A_{t}, \\
_{x}\varepsilon _{|xzt|}^{A}B_{t}\circ f^{\prime \prime }& ={}_{\left\vert
xzt\right\vert }b_{t}^{\left\vert yzt\right\vert }\circ {}_{x}\triangleleft
_{\left\vert yzt\right\vert }^{y}B_{t}\circ (_{x}B_{y}\otimes
{}_{y}\triangleright _{t}^{z}\otimes {}_{y}\triangleleft _{t}^{z})\circ
(_{x}B_{y}\otimes \Delta {}_{\,_{y}B_{z}A_{t}}).
\end{align*}%
Taking into account that $_{x}b_{z}^{y}$ is a coalgebra morphism and using
the definition of $_{x}\triangleright _{z}^{y}$ we get
\begin{equation*}
_{x}A_{|xzt|}\varepsilon _{t}^{B}\circ f^{\prime \prime
}={}_{x}A_{\left\vert xzt\right\vert }\varepsilon _{\left\vert
yzt\right\vert }^{B}\circ {}_{x}\widetilde{R}_{\left\vert yzt\right\vert
}^{y}\circ \left( _{x}B_{y}\otimes (_{y}A_{\left\vert yzt\right\vert
}\varepsilon _{t}^{B}\circ {}_{y}\widetilde{R}_{t}^{z})\right)
={}_{x}\triangleright _{|yzt|}^{y}\circ {}_{x}B_{y}\triangleright _{t}^{z}.
\end{equation*}%
In view of the Remark \ref{re:f=g}, we have $f^{\prime }=f^{\prime \prime }$
if and only if
\begin{equation*}
_{x}A_{|xzt|}\varepsilon _{t}^{B}\circ f^{\prime
}={}_{x}A_{|xzt|}\varepsilon _{t}^{B}\circ f^{\prime \prime }\quad \text{%
and\quad }_{x}\varepsilon _{|xzt|}^{A}B_{t}\text{ }\circ f^{\prime
}={}_{x}\varepsilon _{|xzt|}^{A}B_{t}\circ f^{\prime \prime }.
\end{equation*}%
Thus, if $\left\vert xy\left\vert yzt\right\vert \right\vert =\left\vert
xzt\right\vert ,$ then (\ref{cor1-1}) is equivalent to (\ref{lr1-1})
together with (\ref{lr1-2}). We omit the proof of the second statement,
being similar.

To prove the third part of the lemma we reiterate the above reasoning. We
now take $f^{\prime }$ and $f^{\prime \prime }$ to be the coalgebra morphisms%
\begin{equation*}
f^{\prime }:={}_{x}\widetilde{R}_{y}^{x}\circ \left( 1_{x}^{B}\otimes
{}_{x}A_{y}\right) \quad \text{and\quad }f^{\prime \prime
}:={}_{x}A_{y}{}\otimes 1_{y}^{B}.
\end{equation*}%
Since $\left\vert xxy\right\vert =y$ both $f^{\prime }$and $f^{\prime \prime
}$ target in $_{x}A_{y}B_{y}.$ It is easy to see that (\ref{lr3-1}) together
with (\ref{lr3-2}) are equivalent to (\ref{cor1-3}). Similarly, one shows
that the fourth statement is true.
\end{proof}

\begin{theorem}
\label{thm:mp1}We keep the notation and the assumptions from \S \ref%
{coalg(M')}. The set $\widetilde{R}$ is a simple twisting system in $%
\boldsymbol{{M^{\prime }}}$ if and only if the families $\triangleright $
and ${\triangleleft }$ satisfy the following conditions:
\end{theorem}

\begin{enumerate}
\item[(i)] If $_{x}B_{y}A_{z}$ is not an initial object then $%
_{x}A_{\left\vert xyz\right\vert }B_{z}$ is not an initial object as well.

\item[(ii)] If $_{x}B_{y}B_{z}A_{t}$ is not an initial object in $%
\boldsymbol{{M^{\prime }}},$ then $|xy|yzt||=|xzt|$ and the equations (\ref%
{lr1-1}) and (\ref{lr1-2}) hold.

\item[(iii)] If $_{x}B_{y}A_{z}A_{t}$ is not an initial object in $%
\boldsymbol{{M^{\prime }}},$ then $||xyz|zt|=|xyz|$ and the equations (\ref%
{lr2-1}) and (\ref{lr2-2}) hold.

\item[(iv)] If $_{x}A_{y}$ is not an initial object in $\boldsymbol{{%
M^{\prime }}},$ then $|xxy|=y$ and the equations (\ref{lr3-1}) and (\ref%
{lr3-2}) hold.

\item[(v)] If $_{x}B_{y}$ is not an initial object in $\boldsymbol{{%
M^{\prime }}},$ then $|xyy|=x$ and the equations (\ref{lr4-1}) and (\ref%
{lr4-2}) hold.
\end{enumerate}

\begin{proof}
The condition (i) is a part of the definition of simple twisting systems. If
$_{x}B_{y}B_{z}A_{t}$ is not an initial object in $\boldsymbol{M^{\prime }}$
then we may assume that $|xy|yzt||=|xzt|.$ Thus, by Lemma \ref{le:lr}, the
relation (\ref{cor1-1}) and the equations (\ref{lr1-1}) and (\ref{lr1-2})
are equivalent. To conclude the proof we proceed in a similar way.
\end{proof}

\begin{fact}[Matched pairs and the bicrossed product.]
\label{fa:MP} Let ${\triangleright :=}\{_{x}\triangleright
_{z}^{y}\}_{x,y,z\in S}$ and $\triangleleft :=\{_{x}\triangleleft
_{z}^{y}\}_{x,y,z\in S}$ be two families of maps as in \S \ref{coalg(M')}.
We shall say that the quintuple $(\boldsymbol{A},\boldsymbol{B},{%
\triangleright },\triangleleft ,\left\vert \cdots \right\vert )$ is a\emph{\
matched pair} of $\boldsymbol{M}$-categories if and only if $\triangleright $
and ${\triangleleft }$ satisfy the conditions (i)-(v) from the above
theorem. For a matched pair $(\boldsymbol{A},\boldsymbol{B},{\triangleright }%
,\triangleleft ,\left\vert \cdots \right\vert )$ we have just seen that $(%
\widetilde{R},\left\vert \cdots \right\vert )$ is a simple twisting system
in $\boldsymbol{{M^{\prime }}}$, where $\widetilde{R}:=\{_{x}\widetilde{R}%
_{z}^{y}\}_{x,y,z\in S}$ is the set of coalgebra morphisms which are defined
by the formula (\ref{mp1}). Hence, supposing that $\boldsymbol{{M^{\prime }}}
$ is an $S$-distributive domain, we may construct the twisted tensor product
$\boldsymbol{A}\otimes _{R}\boldsymbol{B},$ which is an enriched category
over $\boldsymbol{{M^{\prime }}}.$ We shall call it the \emph{bicrossed
product} of $(\boldsymbol{A},\boldsymbol{B},{\triangleright },\triangleleft
,\left\vert \cdots \right\vert )$ and we shall denote it by $\boldsymbol{A}%
\Join \boldsymbol{B}.$
\end{fact}

\begin{proposition}
The bicrossed product of a matched pair $(\boldsymbol{A},\boldsymbol{B},{%
\triangleright },\triangleleft ,\left\vert \cdots \right\vert )$ is enriched
over the monoidal category $\boldsymbol{M}:=\boldsymbol{Coalg(M^{\prime })}.$
\end{proposition}

\begin{proof}
Let $\{C_{i}\}_{i\in i}$ be a family of coalgebras in $\boldsymbol{{%
M^{\prime }.}}$ Let us assume that the underlying family of objects has a
coproduct $C:=\textstyle\bigoplus {}_{x\in S}C_{i\text{ }}$ in $\boldsymbol{%
M^{\prime }}.$ Let $\{\sigma _{i}\}_{i\in I}$ by the set of canonical
inclusions into $C.$ There are unique maps $\Delta :C\rightarrow C\otimes C$
and $\varepsilon :C\rightarrow \boldsymbol{1}$ such that
\begin{equation*}
\Delta \circ \sigma _{i}=(\sigma _{i}\otimes \sigma _{i})\circ \Delta
_{i}\quad \text{and\quad }\varepsilon \circ \sigma _{i}=\varepsilon _{i},
\end{equation*}%
for all $i\in I,$ where $\Delta _{i}$ and $\varepsilon _{i}$ are the
comultiplication and the counit of $C_{i}.$ It is easy to see that $%
(C,\Delta ,\varepsilon )$ is a coalgebra in $\boldsymbol{M^{\prime }}.$ Note
that, by the construction of the coalgebra structure on $C,$ the inclusion $%
\sigma _{i}$ is a coalgebra map, for any $i\in I.$ Furthermore, let $%
f_{i}:C_{i}\rightarrow D$ be a coalgebra morphism for every $i\in I.$ By the
universal property of the coproduct there is a unique map $f:C\rightarrow D$
in $\boldsymbol{{M^{\prime }}}$ such that $f\circ \sigma _{i}=f_{i},$ for
all $i\in I.$ It is not difficult to see that $f$ is a morphism of
coalgebras, so $(C,\{\sigma _{i}\}_{i\in I})$ is the coproduct of $%
\{C_{i}\}_{i\in I}$ in $\boldsymbol{M}$.

In particular, $_{x}A_{\overline{u}}B_{y}=\bigoplus_{u\in S}{}_{x}A_{%
\overline{u}}B_{y}$ has a unique coalgebra structure such that the inclusion
$_{x}\sigma _{z}^{u}:{}_{x}A_{u}B_{y}\rightarrow {}_{x}A_{\overline{u}}B_{y}$
is a coalgebra map, for all $x,$ $y$ and $u$ in $S$. Recall that for the
construction of the composition map $_{x}c_{z}^{y}:{}_{x}A_{\overline{u}%
}B_{y}A_{\overline{v}}B_{z}\rightarrow {}_{x}A_{\overline{w}}B_{z}$ one
applies the universal property of the coproduct to $\{f_{u,v}\}_{u,v\in
S^{2}},$ where%
\begin{equation*}
f_{u,v}={}_{x}\sigma _{z}^{\left\vert uyv\right\vert }\circ
{}_{x}a_{\left\vert uyv\right\vert }b_{z}\circ {}_{x}A_{u}\widetilde{R}%
_{v}^{y}B_{z}.
\end{equation*}%
Since $\boldsymbol{A}$ and $\boldsymbol{B}$ are $\boldsymbol{M}$-categories
and $_{u}\widetilde{R}_{v}^{y}$ is a coalgebra map, in view of the foregoing
remarks, it follows that $_{x}c_{z}^{y}$ is a morphism in $\boldsymbol{M},$
for all $x,y,z\in S.$ The identity map of $x$ in $\boldsymbol{A}\Join
\boldsymbol{B}$ is the coalgebra map $_{x}\sigma _{x}^{x}\circ
(1_{x}^{A}\otimes 1_{x}^{B}).$ In conclusion, $\boldsymbol{A}\Join
\boldsymbol{B}$ is enriched over $\boldsymbol{M}.$
\end{proof}

\section{Examples.}

In this section we give some examples of (simple) twisting systems. We start
by considering the case of $\boldsymbol{Set}$-categories, that is usual
categories.

\begin{fact}[Simple twisting systems of enriched categories over $%
\boldsymbol{Set.}$]
\label{fa:sts} The category $\boldsymbol{Set}$ is a braided monoidal
category with respect to the cartesian product, its unit object being $%
\{\emptyset \}$. Clearly, the empty set is the initial object in $%
\boldsymbol{Set},$ and this category is an $S$-distributive domain, for any
set $S.$ We have already noticed that the (\dag ) hypothesis holds in $%
\boldsymbol{Set}.$

Let $\boldsymbol{C}$ be an enriched category over $\boldsymbol{Set}.$ Thus,
by definition, $\boldsymbol{C}$ is a category in the usual sense, that is $%
_{x}{C}_{y}$ is a set for all $x,y\in S.$ An element in $_{x}{C}_{y}$ is
regarded as a function from $y$ to $x$.

It is easy to see that a given set $X$ can be seen in a unique way as a
coalgebra in $\boldsymbol{Set}$. As a matter of fact the comultiplication
and the counit of this coalgebra are given by the diagonal map $\Delta
:X\rightarrow X\times X$ and the constant map $\varepsilon :X\rightarrow
\{\emptyset \},$
\begin{equation*}
\Delta (x)=x\otimes x\text{\quad and\quad }\varepsilon (x)=\emptyset .
\end{equation*}%
Obviously, any function $f:X\rightarrow Y$ is morphism of coalgebras in $%
\boldsymbol{Set}$. Consequently, any category $\boldsymbol{C}$ may be seen
as an enriched category over $\boldsymbol{Coalg(Set)}$,

Our aim is to describe the simple twisting systems between two categories $%
\boldsymbol{B}$ and $\boldsymbol{A}$. In view of the foregoing discussion
and of our results in the previous section, for any simple twisting system $%
R:=\{_{x}\widetilde{R}_{z}^{y}\}_{x,y,z\in S}$ there is a unique matched
pair $(\boldsymbol{A},\boldsymbol{B},\triangleright ,{\triangleleft }%
,\left\vert \cdots \right\vert ),$ and conversely. These structures are
related each other by the formulae (\ref{lr}) and (\ref{mp1}).

Since $\boldsymbol{A}$ is an usual category, the composition of morphisms
will be denoted in the traditional way $g\circ g^{\prime },$ for any $g\in
{}_{x}A_{y}$ and $g^{\prime }\in {}_{y}A_{z}$ (recall that the domain and
the codomain of $g$ are $y$ and $x,$ respectively). The same notation will
be used for $\boldsymbol{B}.$ On the other hand, for any $f\in {}_{x}B_{y}$
and $g\in {}_{y}A_{z}$ we shall write%
\begin{equation*}
f\triangleright g:={}_{x}\triangleright {}_{z}^{y}(f,g)\quad \text{and}\quad
f\triangleleft g={}_{x}{\triangleleft }_{z}^{y}(f,g).
\end{equation*}%
Since the comultiplication in this case is always the diagonal map, and the
counit is the constant map to $\{\emptyset \}$, the conditions of Theorem %
\ref{thm:mp1} and the following ones are equivalent.

\begin{enumerate}
\item[(i)] If $_{x}B_{y}A_{z}$ is not empty then $_{x}A_{\left\vert
xyz\right\vert }B_{z}$ is not empty as well.

\item[(ii)] For any $(f,f^{\prime },g)\in {}_{x}B_{y}B_{z}A_{t}$ we have $%
|xy|yzt||=|xzt|$, and%
\begin{equation*}
(f\circ f^{\prime })\triangleright g=f\triangleright (f^{\prime
}\triangleright g)\qquad \text{and}\qquad (f\circ f^{\prime })\triangleleft
g=[f\triangleleft (f^{\prime }\triangleright g)]\circ (f^{\prime
}\triangleleft g).
\end{equation*}

\item[(iii)] For any $(f,g,g^{\prime })\in {}_{x}B_{y}A{}_{z}A_{t}$ we have $%
||xyz|zt|=|xyt|$, and%
\begin{equation*}
f\triangleleft (g\circ g^{\prime })=(f\triangleleft g)\triangleleft
g^{\prime }\qquad \text{and}\qquad f\triangleright (g\circ g^{\prime
})=(f\triangleright g)\circ \lbrack (f\triangleleft g)\triangleright
g^{\prime }].
\end{equation*}

\item[(iv)] For any $g\in {}_{x}A_{y}$ we have $|xxy|=y$, and%
\begin{equation*}
1_{x}^{B}\triangleright g=g\text{\qquad and\qquad}1_{x}^{B}\triangleleft
g=1_{y}^{B}.  \label{set5}
\end{equation*}

\item[(v)] For any $f\in {}_{x}B_{y}$ we have $|xyy|=x$, and%
\begin{equation*}
f\triangleright 1_{y}^{A}{}=1_{x}^{B}\text{\qquad and\qquad}f\triangleleft
1_{y}^{A}=f.  \label{set6}
\end{equation*}
\end{enumerate}

In this case the bicrossed product $\boldsymbol{A}\Join \boldsymbol{B}$ is
the category whose hom-sets are given by
\begin{equation*}
_{x}(\boldsymbol{A}\Join \boldsymbol{B})_{y}=\textstyle\coprod\limits_{u\in
S}{}_{x}{A}_{u}{B}_{y}.
\end{equation*}%
The identity of $x$ in $\boldsymbol{A}\Join \boldsymbol{B}$ is $%
(1_{x}^{A},1_{x}^{B}).$ For $(g,f)\in $ $_{x}{A}_{u}{B}_{y}$ and $(g^{\prime
},f^{\prime })\in {}_{y}{A}{}_{v}{B}_{z}$ we have%
\begin{equation*}
(g,f)\circ (g^{\prime },f^{\prime })=\left( g\circ (f\triangleright
g^{\prime }),(f\triangleleft g^{\prime })\circ f^{\prime }\right) .
\end{equation*}
\end{fact}

\begin{remark}
R. Resebrugh and R.J. Wood showed that every twisting systems between two
$\boldsymbol{Set}$-categories $\boldsymbol{B}$ and $\boldsymbol{A}$ is completely determined by
a left action $\triangleright $ of $\boldsymbol{B}$ on $\boldsymbol{A}$ and
a right action $\triangleleft $ of $\boldsymbol{A}$ on $\boldsymbol{B.}$
More precisely, given a twisting system $R=\{_{x}R_{z}^{y}\}_{x,y,z\in S}$
and the morphisms $f\in {}_{x}B_{y}$ and $g\in {}_{y}A_{z},$ then $%
_{x}R_{z}^{y}(f,g)$ is an element in $_{x}A_{u}B_{z},$ where $u$ is a
certain element of $S.$ Hence, there are unique morphisms $f\triangleright
g\in {}_{x}A_{u}$ and $f\triangleleft g\in {}_{u}B_{z}$ such that
\begin{equation*}
_{x}R_{z}^{y}(g,f)=\left( f\triangleright g,f\triangleleft g\right) .
\end{equation*}%
The actions $\triangleright $ and
$\triangleleft $ must satisfy several compatibility conditions, which are similar
to those that appear in the above characterization of simple twisting systems.
For details the reader is referred to the second section of  \cite{RW}.
\end{remark}

\begin{fact}[The bicrossed product of two groupoids.]
\label{fa:gr} We now assume that $(\boldsymbol{A},\boldsymbol{B}%
,\triangleright ,{\triangleleft ,}|{\cdots }|)$ is a matched pair of
groupoids. Recall that a groupoid is a category whose morphisms are
invertible. We claim that $\boldsymbol{A}\Join \boldsymbol{B}$ is also a
groupoid. Indeed, as in the case of monoids, one can show that a category is
a groupoid if and only if every morphism is right invertible (or left
invertible). Since%
\begin{equation*}
_{x}(\boldsymbol{A}\Join \boldsymbol{B})_{y}=\textstyle\coprod\limits_{u\in
S}{_{x}{A}}_{u}{B}_{y},
\end{equation*}%
it is enough to prove that $(g,f)$ is right invertible, where $g\in
{}_{x}A_{u}$ and $f\in {}_{u}B_{y}$ are arbitrary morphisms. Therefore, we
are looking for a pair $(g^{\prime },f^{\prime })\in {}_{y}A_{v}\times $ $%
{}_{v}B_{x}$ such that
\begin{equation*}
g\circ (f\triangleright g^{\prime })=1_{x}^{A}\qquad \text{and}\qquad
(f\triangleleft g^{\prime })\circ f^{\prime }=1_{x}^{B}.
\end{equation*}%
Since $g$ is an invertible morphism in $_{x}A_{u}$ we get that $%
f\triangleright g^{\prime }=g^{-1}\in {}_{u}A_{x}.$ Since $f$ is invertible
too,
\begin{equation*}
g^{\prime }=1_{y}^{A}\triangleright g^{\prime }=(f^{-1}\circ
f)\triangleright g^{\prime }=f^{-1}\triangleright (f\triangleright g^{\prime
})=f^{-1}\triangleright g^{-1}.
\end{equation*}%
As $g^{\prime }\in {}_{y}A_{v}$ and $f^{-1}\triangleright g^{-1}\in
{}_{y}A_{|yux|}$ we must have $v=|yux|$. Thus we can now take%
\begin{equation*}
f^{\prime }=[f\triangleleft (f^{-1}\triangleright g^{-1})]^{-1}\in
{}_{|yux|}B_{x}.
\end{equation*}
\end{fact}

\begin{fact}[The smash product category.]
We take $\boldsymbol{M}$ to be the monoidal category $\mathbb{K}$-$%
\boldsymbol{Mod},$ where $\mathbb{K}$ is a commutative ring. Hence in this
case we work with $\mathbb{K}$-linear categories. Let $H$ be a $\mathbb{K}$%
-bialgebra. We define an enriched category $\boldsymbol{H}$ over $\mathbb{K}$%
-$\boldsymbol{Mod}$ by setting $_{x}H_{x}=H$ and $_{x}H_{y}=0,$ for $x\neq
y. $ The composition of morphisms in $\boldsymbol{H}$ is given by the
multiplication in $H$ and the identity of $x$ is the unit of $H$. For the
comultiplication of $H$ we shall use the $\Sigma $-notation%
\begin{equation*}
\Delta (h)=\textstyle\sum {}h_{(1)}\otimes h_{(2)}.
\end{equation*}%
Let $\boldsymbol{A}$ denote an $H$-module category, i.e. a category enriched
in $H$-$\boldsymbol{Mod}$. Thus $H$ acts on $_{x}A_{y}$, for any $x,y\in S$,
and the composition and the identity maps in $\boldsymbol{A}$ are $H$-linear
morphisms. Obviously, $\boldsymbol{A}$ is a $\mathbb{K}$-linear category.
Our aim is to associate to $\boldsymbol{A}$ a simple twisting system ${R}%
=\{_{x}\widetilde{R}_{z}^{y}\}_{x,y,z\in S}$. First we define $|\cdots
|:S^{3}\rightarrow S$ by $|xyz|=z.$ Then, using the actions $\cdot :H\otimes
{}_{x}A_{z}\rightarrow {}_{x}A_{z\text{,}}$ we define
\begin{equation*}
_{x}\widetilde{R}_{z}^{x}:H\otimes {}_{x}A_{z}\rightarrow {}_{x}A_{z}\otimes
H,\quad _{x}\widetilde{R}_{z}^{x}(h\otimes f)=\textstyle\sum {}h_{(1)}\cdot
f\otimes h_{(2)}.
\end{equation*}%
For $x\neq y$ we take $_{x}\widetilde{R}_{z}^{y}=0$. It is easy to see that $%
{R}$ is a simple twisting system of $\mathbb{K}$-linear categories. Clearly $%
\mathbb{K}$-$\boldsymbol{Mod}$ is $S$-distributive, for any set $S.$ If $%
\mathbb{K}$ is a field then $\mathbb{K}$-$\boldsymbol{Mod}$ is a domain, so
in this case the twisted tensor product of $\boldsymbol{A}$ and $\boldsymbol{%
H}$ with respect to $R$ makes sense, cf. Theorem \ref{thm:TTP}. It is called
the smash product of $\boldsymbol{A}$ by $H$, and it is denoted by $%
\boldsymbol{A}\#H.$ By definition, $_{x}(\boldsymbol{A}\#H)_{y}={}_{x}A_{y}%
\otimes H$ and
\begin{equation}
(f\otimes h)\circ (f^{\prime }\otimes h^{\prime })=\textstyle\sum {}f\circ
(h_{(1)}\mathbf{\cdot }\text{ }f^{\prime })\otimes h_{(2)}h^{\prime },
\label{smash}
\end{equation}%
for any $f\in {}_{x}A_{y}$, $f^{\prime }\in {}_{y}A_{z}$ and $h,h^{\prime
}\in H.$
\end{fact}

\begin{fact}[The semidirect product.]
\label{sdp} Let $\boldsymbol{A}$ be a category. Let us suppose that $({B}%
,\cdot ,{1})$ is a monoid that acts to the left on each $_{x}A_{y}$ via
\begin{equation*}
\triangleright :B\times {}_{x}A_{y}\rightarrow {}_{x}A_{y}.
\end{equation*}%
We define the category $\boldsymbol{B}$ so that $_{x}B_{x}=B$ and ${}_{x}{B}%
_{y}=\emptyset $, for $x\neq y$. The composition of morphisms is given by
the multiplication in $\boldsymbol{B}.$ To this data we associate a matched
pair $(\boldsymbol{A},\boldsymbol{B},|{\cdots }|,\triangleright
,\triangleleft ),$ setting $f\triangleleft g=f$ for any $(f,g)\in
{}_{x}B_{x}A_{y},$ and defining the function $\left\vert \cdots \right\vert
:S^{3}\rightarrow S$ by $|xyz|=z.$ Note that if $x\neq y$ then $%
_{x}B{}_{y}A_{z}$ is empty, so $_{x}\triangleright _{z}^{y}$ and $_{x}{%
\triangleleft }_{z}^{y}$ coincide with the empty function. One shows that $(%
\boldsymbol{A},\boldsymbol{B},\triangleright ,\triangleleft ,|{\cdots }|)$
is a matched pair if and only if for any $(g,g^{\prime })\in
{}_{x}A_{y}A_{z} $ and $f\in B$%
\begin{align*}
f\triangleright (g\circ g^{\prime })& =(f\triangleright g)\circ
(f\triangleright g^{\prime }), \\
f\triangleright 1_{x}^{A}& =1_{x}^{A}.
\end{align*}%
The corresponding bicrossed product will be denoted in this case by $%
\boldsymbol{A\rtimes }B$. If $\left\vert S\right\vert =1$ then $\boldsymbol{A%
}$ can be identified with a monoid and $\boldsymbol{A}\rtimes B$ is the
usual semidirect product of two monoids. For this reason we shall call $%
\boldsymbol{A}\rtimes B$ the semidirect product of $\boldsymbol{A}$ with $B.$
Note that $_{x}(\boldsymbol{A}\rtimes B)_{y}={}_{x}A_{y}\times {}B.$ For any
$f,f^{\prime }\in B$ and $(g,g^{\prime })\in {}_{x}A_{y}A_{z}$, the
composition of morphisms in $\boldsymbol{A}\rtimes B$ is given by
\begin{equation*}
(g,f)\circ (g^{\prime },f^{\prime })=(g\circ (f\triangleright g^{\prime
}),f\circ f^{\prime }).
\end{equation*}
\end{fact}

\begin{fact}[Twisting systems between algebras in $\boldsymbol{M}.$]
We now consider a twisting system ${R}$ between two $\boldsymbol{M}$%
-categories $\boldsymbol{B}$ and $\boldsymbol{A}$ with the property that ${S}%
=\{x_{0}\}$. Obviously, $\boldsymbol{M}$ is $S$-distributive. We shall use
the notation $A={}_{x_{0}}A_{x_{0}}$ and $B={}_{x_{0}}B_{x_{0}}.$ The
composition map $a:={}_{x_{0}}a_{x_{0}}^{x_{0}}$ and $1_{A}:=1_{x_{0}}^{A}$
define an algebra structure on $A$. A similar notation will be used for the
algebra corresponding to the $\boldsymbol{M}$-category $\boldsymbol{B}.$ Let
$R$ be a morphism from $B\otimes A$ to $A\otimes B.$

Since $_{x_{0}}\sigma _{x_{0}}^{x_{0}}$ is the identity map of $B\otimes A,$
by Proposition \ref{thm R simplu}, we deduce that $%
{}_{x_{0}}R_{x_{0}}^{x_{0}}=R$ defines a twisting system between $%
\boldsymbol{B}$ and $\boldsymbol{A}$ if and only if $R$ satisfies the
relations (\ref{cor1-1})-(\ref{cor1-4}) with respect to the unique map $%
\left\vert \cdots \right\vert :S^{3}\rightarrow S.$ In turn, they are
equivalent to the following identities%
\begin{align}
R\circ (b\otimes A)& =(A\otimes b)\circ (R\otimes B)\circ (B\otimes R),
\label{fm1} \\
R\circ (B\otimes a)& =(a\otimes B)\circ (A\otimes R)\circ (R\otimes A),
\label{fm2} \\
R\circ (1_{B}^{{}}\otimes A)& =A\otimes 1_{B}^{{}},  \label{fm3} \\
R\circ (B\otimes 1_{A})& =1_{A}\otimes B.  \label{fm4}
\end{align}%
In conclusion, in the case when $\left\vert S\right\vert =1,$ to give a
twisting system between $\boldsymbol{B}$ and $\boldsymbol{A}$ is equivalent
to give a \emph{twisting map} between the algebras $B$ and $A,$ that is a
morphism $R$ which satisfies (\ref{fm1})-(\ref{fm4}).

By applying Corollary \ref{cor: TTP} to a twisting map $R:B\otimes
A\rightarrow A\otimes B$ (viewed as a twisting system between two $%
\boldsymbol{M}$-categories with one object) we get the \emph{twisted tensor
algebra }$A\otimes _{R}B.$ The unit of this algebra is $1_{A}\otimes 1_{B}$
and its multiplication is given by
\begin{equation*}
m=(a\otimes b)\circ (A\otimes R\otimes B).
\end{equation*}%
Note that, in view of the foregoing remarks, an algebra $C$ in $\boldsymbol{M%
}$ factorizes through $A$ and $B$ if and only if it is isomorphic to a
twisted tensor algebra $A\otimes _{R}B,$ for a certain twisting map $R.$

An algebra in the monoidal category $\mathbb{K}$-$\boldsymbol{Mod}$ is by
definition an associative and unital $\mathbb{K}$-algebra. Twisted tensor $%
\mathbb{K}$-algebras were investigated for instance in \cite{Ma1}, \cite{Tam}%
, \cite{CSV}, \cite{CIMZ}, \cite{LPvO} and \cite{JLPvO}.

Coalgebras over a field $\mathbb{K}$ are algebras in the monoidal category $(%
\mathbb{K}$-$\boldsymbol{Mod})^{o}.$ Hence a twisting map between two
coalgebras $(A,\Delta _{A},\varepsilon _{A})$ and $(B,\Delta
_{B},\varepsilon _{B})$ is a $\mathbb{K}$-linear map
\begin{equation*}
R:A\otimes _{\mathbb{K}}B\rightarrow B\otimes _{\mathbb{K}}A
\end{equation*}%
which satisfies the equations that are obtained from \eqref{fm1}-\eqref{fm4}
by making the substitutions $a:=\Delta _{A},$ $b:=\Delta _{B},$ $%
1_{A}:=\varepsilon _{A}$ and $1_{B}:=\varepsilon _{B},$ and reversing the
order of the factors with respect to the composition in $\boldsymbol{M}$.
For example \eqref{fm1} should be replaced with%
\begin{equation*}
(\Delta _{B}\otimes _{\mathbb{K}}A)\circ R=(B\otimes _{\mathbb{K}}R)\circ
(R\otimes _{\mathbb{K}}B)\circ (A\otimes _{\mathbb{K}}\Delta _{B}).
\end{equation*}%
Obviously $A\otimes _{R}B$ is the $\mathbb{K}$-coalgebra $(A\otimes _{%
\mathbb{K}}B,\Delta ,\varepsilon )$, where $\varepsilon :=\varepsilon
_{A}\otimes $ $\varepsilon _{B}$ and%
\begin{equation*}
\Delta =(A\otimes R\otimes B)\circ (\Delta _{A}\otimes _{\mathbb{K}}\Delta
_{B}).
\end{equation*}

An algebra in $\Lambda $-$\boldsymbol{Mod}$-$\Lambda $ is called a $\Lambda $%
-\emph{ring}. Specializing $\boldsymbol{M}$ to $\Lambda $-$\boldsymbol{Mod}$-%
$\Lambda $ we find the definition of the \emph{twisted tensor }$\Lambda $%
\emph{-ring}. Dually, $\Lambda $-corings are algebras in $(\Lambda $-$%
\boldsymbol{Mod}$-$\Lambda )^{o}$. Thus in this particular case we are led
to the construction of the \emph{twisted tensor} $\Lambda $-\emph{coring}.

By definition a monad on a category $\boldsymbol{A}$ is an algebra in $[%
\boldsymbol{A},\boldsymbol{A}]$. If $(F,\mu _{F},\iota _{F})$ and $(G,\mu
_{G},\iota _{G})$ are monads in $\boldsymbol{M}$, then a natural
transformation%
\begin{equation*}
\lambda :G\circ F\rightarrow F\circ G
\end{equation*}%
satisfies the relations \eqref{fm1}-\eqref{fm4} if and only if $\lambda $ is
a \emph{distributive law}, cf. \cite{Be}. Let $F^{2}:=F\circ F.$ For every
distributive law $\lambda $ we get a monad $(F\circ G,\mu ,\iota ),$ where $%
\iota :=\iota _{F}G\circ \iota _{G}=F\iota _{G}\circ \iota _{F}$ and
\begin{equation*}
\mu :=\mu _{F}G\circ F^{2}\mu _{G}\circ F\lambda G=F\mu _{G}\circ \mu
_{F}G^{2}\circ F\lambda G.
\end{equation*}%
Distributive laws between comonads can be defined similarly, or working in $[%
\boldsymbol{A},\boldsymbol{A}]^{o}$.

Finally, twisting maps in $\boldsymbol{Opmon(\boldsymbol{M})}$ have been
considered in \cite{BV}. In loc. cit. the authors define a bimonad in $%
\boldsymbol{M}$ as an algebra in $\boldsymbol{Opmon(\boldsymbol{M}})$. Hence
a twisting map between two bimonads is an opmonoidal distributive law
between the underlying monads. For any opmonoidal distributive law $\lambda $
between the bimonads $G$ and $F$, there is a canonical bimonad structure on
the endofunctor $F\circ G$. See \cite[Section 4]{BV} for details.
\end{fact}

\begin{fact}[Matched pairs of algebras in $\boldsymbol{Coalg({M^{\prime }}){.%
}}$]
Let $\boldsymbol{M}$ denote the category of coalgebras in a braided monoidal
category $(\boldsymbol{{M^{\prime },}}\otimes ,\mathbf{1},\chi ).$ By
definition, a bialgebra in $\boldsymbol{M^{\prime }}$ is an algebra in $%
\boldsymbol{M}$. We fix two bialgebras $(A,a,1_{A},\Delta_A,\varepsilon_A)$
and $(A,b,1_{B},\Delta_B,\varepsilon_B)$ in $\boldsymbol{M^{\prime }}$ and
take $R:B\otimes A\rightarrow A\otimes B $ to be a morphism in $\boldsymbol{M%
}.$ By Lemma \ref{le:coalgebra}, there are the coalgebra maps $%
\triangleright :B\otimes A\rightarrow A$ and ${\triangleleft }:B\otimes
A\rightarrow B$ such that
\begin{equation}
R=(\triangleright \otimes {\triangleleft })\circ \Delta _{B\otimes A}\qquad
\text{and\qquad }\chi _{A,B}\circ (\triangleright \otimes {\triangleleft }%
)\circ \Delta _{B\otimes A}=({\triangleleft }\otimes \triangleright )\circ
\Delta _{B\otimes A}.  \label{bi1}
\end{equation}%
We have seen that $R$ is a twisting map in $\boldsymbol{M}$ if and only if
it satisfies the relations (\ref{fm1})-(\ref{fm4}). In view of Lemma \ref%
{le:lr}, these equations are equivalent to the fact that $(A,\triangleright
) $ is a left $B$-module and $(B,{\triangleleft })$ is a right $A$-module
such that the following identities hold:%
\begin{align}
{\triangleleft }\circ (b\otimes A)& =b\circ ({\triangleleft }\otimes B)\circ
(B\otimes \triangleright \otimes {\triangleleft })\circ (B\otimes \Delta
_{B\otimes A}),  \label{bi2} \\
\triangleright \circ (B\otimes a)& =a\circ (A\otimes \triangleright )\circ
(\triangleright \otimes {\triangleleft }\otimes A)\circ (\Delta _{B\otimes
A}\otimes A),  \label{bi3} \\
{\triangleleft }\circ (1_{B}\otimes A)& =\varepsilon _{A}\otimes 1_{B},
\label{bi4} \\
\triangleright \circ (B\otimes 1_{A})& =1_{A}\otimes \varepsilon _{B}.
\label{bi5}
\end{align}%
By definition, a \emph{matched pair of bialgebras} \emph{in} $\boldsymbol{{%
M^{\prime }}}$ consists of a left $B$-action $(A,\triangleright )$ and a
right $A$-action $(B,{\triangleleft })$ in $\boldsymbol{M}$ such that the
second equation in (\ref{bi1}) and the relations (\ref{bi2})-(\ref{bi5})
hold. For a matched pair of bialgebras we shall use the notation $%
(A,B,\triangleright ,{\triangleleft ).}$

Summarizing, there is an one to-one-correspondence between twisting maps of
bialgebras in $\boldsymbol{{M^{\prime }}}$ and matched pairs of bialgebras
in $\boldsymbol{{M^{\prime }}}.$ If $(A,B,\triangleright ,{\triangleleft )}$
is a matched pair of bialgebras and $R$ is the corresponding twisting map,
then $A\otimes _{R}B$ will be called the \emph{bicrossed product of the
bialgebras }$A$ and $B,$ and it will be denoted by $A\Join B.$ Note that $%
A\Join B$ is an algebra in $\boldsymbol{M}.$ Thus the bicrossed product of $%
A $ and $B$ is a bialgebra in $\boldsymbol{{M^{\prime }}}.$ The unit of this
bialgebra is $1_{A}\otimes 1_{B}$ and the multiplication is given by%
\begin{equation*}
m=(a\otimes b)\circ (A\otimes \triangleright \otimes {\triangleleft }\otimes
B)\circ (A\otimes \Delta _{B\otimes A}\otimes B).
\end{equation*}%
As a coalgebra $A\Join B$ is the tensor product coalgebra of $A$ and $B.$ We
also conclude that a bialgebra $C$ in $\boldsymbol{{M^{\prime }}}$
factorizes through the sub-bialgebras $A$ and $B$ if and only if $C\cong
A\Join B$.

As a first application, let us take $\boldsymbol{M^{\prime }}$ to be the
category of sets, which is braided with respect to the braiding given by $%
(X,Y)\mapsto (Y,X)$ and $(f,g)\mapsto (g,f)$, for any sets $X,Y$ and any
functions $f,g.$ We have already noticed that there is a unique coalgebra
structure on a given set $X$%
\begin{equation*}
\Delta (x)=(x,x),\quad \quad \varepsilon (x)=\emptyset ,
\end{equation*}%
where $\emptyset $ denotes the empty set; recall that the unit object in $%
\boldsymbol{Set}$ is $\{\emptyset \}.$ Hence an ordinary monoid, i.e. an
algebra in $\boldsymbol{Set},$ has a natural bialgebra structure in this
braided category. Moreover any twisting map $R:B\times A\rightarrow A\times
B $ between two monoids $(A,\cdot ,1_{A})$ and $(B,\cdot ,1_{B})$ is a
twisting map of bialgebras in $\boldsymbol{Set}$. Let $(A,B,\triangleright ,{%
\triangleleft )}$ be the corresponding matched pair. One easily shows that
the second condition in (\ref{bi1}) is always true. By notation, the
functions $\triangleright $ and ${\triangleleft }$ map $(f,g)\in B\times A$
to $f\triangleright g$ and $f{\triangleleft }g,$ respectively. Hence the
equations (\ref{bi2})-(\ref{bi5}) are equivalent to the following ones:
\begin{gather*}
(f\cdot f^{\prime })\triangleleft g=[f\triangleleft (f^{\prime
}\triangleright g)]\cdot (f^{\prime }\triangleleft g), \\
f\triangleright (g\cdot g^{\prime })=(f\triangleright g)\cdot \lbrack
(f\triangleleft g)\triangleright g^{\prime }], \\
1_{B}\triangleleft g=1_{B}\qquad \text{and\qquad }f\triangleright
1_{A}=1_{A}.
\end{gather*}%
Since $R(f,g)=\left( f\triangleright g,f\triangleleft g\right) $ the product
of the monoid $A\Join B$ is defined by the formula%
\begin{equation*}
\left( g,f\right) \cdot \left( g^{\prime },f^{\prime }\right) =\left(
g\left( f\triangleright g^{\prime }\right) ,\left( f\triangleleft g^{\prime
}\right) f^{\prime }\right) .
\end{equation*}%
In conclusion, a monoid $C$ factorizes through $A$ and $B$ if and only if $%
C\cong A\Join B.$

We have seen that the bicrossed product of two groupoids is a groupoid.
Thus, if $A$ and $B$ are groups, then $A\Join B$ is a group as well. This
result was proved by Takeuchi who introduced the matched pairs of groups in
\cite{Tak}.

We now consider the braided category $\mathbb{K}$-$\boldsymbol{Mod}$, whose
braiding is the usual flip map. An algebra in $\boldsymbol{M}$, the monoidal
category of $\mathbb{K}$-coalgebras, is a bialgebra over the ring $\mathbb{K}
$, and conversely. Proceeding as in the previous case, one shows that a
twisting map $R:B\otimes _{\mathbb{K}}A\rightarrow A\otimes _{\mathbb{K}}B$
of bialgebras is uniquely determined by the coalgebra maps $\triangleright
:B\otimes _{\mathbb{K}}A\rightarrow A$ and ${\triangleleft :B}\otimes _{%
\mathbb{K}}A\rightarrow B$ via the formula
\begin{equation*}
R(f\otimes g)=\textstyle\sum {}(f_{(1)}\triangleright g_{(1)})\otimes
(f_{(2)}\triangleleft g_{(2)}).
\end{equation*}
Using the $\Sigma $-notation, the second equation in (\ref{bi1}) is true if
and only if
\begin{equation*}
\textstyle\sum {}(f_{(1)}\triangleleft g_{(1)})\otimes
(f_{(2)}\triangleright g_{(2)})=\textstyle\sum {}(f_{(2)}\triangleleft
g_{(2)})\otimes (f_{(1)}\triangleright g_{(1)}),
\end{equation*}%
for any $f\in B$ and $g\in A.$ On the other hand, the equations (\ref{bi2})-(%
\ref{bi5}) hold if and only if
\begin{align*}
(gg^{\prime })\triangleleft f& =\textstyle\sum {}[g\triangleleft
(g_{(1)}^{\prime }\triangleright f_{(1)})](g_{(2)}^{\prime }\triangleleft
f_{(2)}), \\
g\triangleright (ff^{\prime })& =\textstyle\sum {}(g_{(1)}\triangleright
f_{(1)})[(g_{(2)}\triangleleft f_{(2)})\triangleright f^{\prime }], \\
f\triangleright 1^{A}& =\varepsilon _{B}(f)1^{A}, \\
1^{B}\triangleleft g& =\varepsilon _{A}(g)1^{B},
\end{align*}%
for any $f,f^{\prime }\in B$ and any $g,g^{\prime }\in A.$ Thus, we
rediscover the definition of \emph{matched pairs of bialgebras} and the
formula for the multiplication of the \emph{double cross product}, see \cite[%
Theorem 7.2.2]{Ma2}. Namely,
\begin{equation*}
(f\otimes g)(f^{\prime }\otimes g^{\prime })=\textstyle\sum
{}f(g_{(1)}\triangleright f_{(1)}^{\prime })\otimes (g_{(2)}\triangleleft
f_{(2)}^{\prime })g^{\prime }.
\end{equation*}
\end{fact}

\begin{fact}[Twisting systems between thin categories.]
\label{fa:thin}Our aim now is to investigate the twisting systems between
two thin categories $\boldsymbol{B}$ and $\boldsymbol{A}.$ Recall that a
category is thin if there is at most one morphism between any couple of
objects. Thus, for any $x$ and $y$ in $S$ we have that either $%
_{x}A_{y}=\{_{x}g_{y}\}$ or $_{x}A_{y}$ is the empty set. Clearly, if $%
_{x}A_{y}=\{_{x}g_{y}\}$ and $_{y}A_{z}=\{_{y}g_{z}\}$ then $_{x}g_{y}\circ
{}_{y}g_{z}={}_{x}g_{z}$. The identity morphism of $x$ is $_{x}g_{x}.$
Similarly, if $_{x}B_{y}$ is not empty then $_{x}B_{y}=\{_{x}f_{y}\}.$

We fix a twisting system $R$ between $\boldsymbol{B}$ and $\boldsymbol{A.}$
It is defined by a family of maps
\begin{equation*}
_{x}R_{z}^{y}:{}_{x}B_{y}\times {}_{y}A_{z}\rightarrow \textstyle%
\coprod\limits_{u\in S}A_{u}B_{z}
\end{equation*}%
that render commutative the diagrams in Figure \ref{fig:R_1}. We claim that $%
R$ is simple. We need a function $\left\vert \cdots \right\vert
:S^{3}\rightarrow S$ such that the image of $_{x}R_{z}^{y}$ is included into
$_{x}A_{|xyz|}B_{z}$ for all $(x,y,z)\in S^{3}.$ Let $T\subseteq S^{3}$
denote the set of all triples such that ${}_{x}B_{y}A_{z}={}_{x}B_{y}\times
{}_{y}A_{z}$ is not empty. Of course, if $(x,y,z)$ is not in $T$ then $%
_{x}R_{z}^{y}$ is the empty function, so we can take $\left\vert
xyz\right\vert $ to be an arbitrary element in $S.$ For $(x,y,z)\in T$ there
exists $\left\vert xyz\right\vert \in S$ such that
\begin{equation}
_{x}R_{z}^{y}(_{x}f_{y},{}_{y}g_{z})=(_{x}g_{\left\vert xyz\right\vert
},{}_{\left\vert xyz\right\vert }f_{z}).  \label{ec:Rs}
\end{equation}%
Hence $_{x}R_{z}^{y}$ is a function from $_{x}B_{y}A_{z}$ to $%
_{x}A_{|xyz|}B_{z}.$ Note that $_{x}A_{|xyz|}B_{z}$ is not empty in this
case. For any $(x,y,z)\in S^{3}$ we set $_{x}\widetilde{R}%
_{z}^{y}:={}_{x}R_{z}^{y}.$ By Proposition \ref{thm R simplu} and Corollary %
\ref{cor: R simplu1} it follows that $R$ is simple. We would like now to
rewrite the conditions from the definition of simple twisting systems in an
equivalent form, that only involves properties of $T$ and $\left\vert \cdots
\right\vert .$ For instance let us show that the first condition from
Corollary \ref{cor: R simplu1} is equivalent to:

\begin{enumerate}
\item[(i)] If $(y,z,t)\in T$ and $(x,y,\left\vert yzt\right\vert )\in T\ $%
then $|xy|yzt||=|xzt|.$
\end{enumerate}

Indeed, if $_{x}B_{y}B_{z}A_{t}\ $is not empty then $_{y}B_{z}A_{t}\neq
\emptyset ,$ so $(y,z,t)\in T.$ We have already noticed that $%
_{y}A_{|yzt|}B_{t}$ is not empty, provided that $_{y}B_{z}A_{t}\ $is so.
Since $_{x}B_{y}$ and $_{y}A_{|yzt|}$ are not empty it follows that $%
_{x}B_{y}A_{|yzt|}$ has the same property, that is $(x,y,\left\vert
yzt\right\vert )\in T.$ Therefore, if $_{x}B_{y}B_{z}A_{t}$ is not empty
then $(y,z,t)\in T$ and $(x,y,\left\vert yzt\right\vert )\in T.$ It is easy
to see that the reversed implication is also true. Thus it remains to prove
that the equation \eqref{cor1-1} holds in the case when $_{x}B_{y}B_{z}A_{t}$
is not empty. But this is obvious, as $_{x}R_{t}^{z}{}\circ
{}_{x}b_{z}^{y}A_{t}$\ and $({}_{x}A_{|xy|yzt||}b_{t}^{|yzt|})\circ
{}_{x}R_{|yzt|}^{y}B_{t}\circ {}_{x}B_{y}R_{t}^{z}$ have the same source $%
_{x}B_{y}B_{z}A_{t}$ and the same target $%
_{x}A_{|xzt|}B_{t}={}_{x}A_{|xy|xyz||}B_{t}$. Both sets are singletons, so
the above two morphisms must be equal.

Proceeding in a similar way, we can prove that the other three conditions
from Corollary \ref{cor: R simplu1} are respectively equivalent to:

\begin{enumerate}
\item[(ii)] If $(x,y,z)\in T$ and $(\left\vert xyz\right\vert ,z,t)\in T$
then $\left\vert \left\vert xyz\right\vert zt\right\vert =\left\vert
xyt\right\vert .$

\item[(iii)] If $(x,x,y)\in T$ then $|xxy|=y.$

\item[(iv)] If $(x,y,y)\in T$ then $|xyy|=x.$
\end{enumerate}

The last condition in the definition of simple twisting systems is
equivalent to:

\begin{enumerate}
\item[(v)] If $(x,y,z)\in T$ then $_{x}A_{|xyz|}B_{z}$ is not empty.
\end{enumerate}

Hence for a twisting system $R$ the function $\left\vert \cdots \right\vert $
satisfies the above five conditions. Conversely, let $\left\vert \cdots
\right\vert :S^{3}\rightarrow S$ denote a function such that the above five
conditions hold. Let $_{x}R_{z}^{y}$ be the empty function, if $(x,y,z)$ is
not in $T.$ Otherwise we define $_{x}R_{z}^{y}$ by the formula (\ref{ec:Rs}%
). In view of the foregoing remarks it is not difficult to see that $%
R:=\{_{x}R_{z}^{y}\}_{x,y,z\in S}$ is a simple twisting system.

Clearly two functions $\left\vert \cdots \right\vert $ and $\left\vert
\cdots \right\vert ^{\prime }$ induce the same twisting system if and only
if their restriction to $T$ are equal. Summarizing, we have just proved the
theorem below.
\end{fact}

\begin{theorem}
\label{thm:thin}Let $\boldsymbol{A}$ and $\boldsymbol{B}$ be thin
categories. Let $T$ denote the set of all triples $(x,y,z)\in S^{3}$ such
that $_{x}B_{y}A_{z}$ is not empty. If $R$ is a twisting system between $%
\boldsymbol{B}$ and $\boldsymbol{A}$ then there exists a function $%
\left\vert \cdots \right\vert :S^{3}\rightarrow S$ such that the conditions
(i)-(v) from the previous subsection hold, and conversely. Two functions $%
\left\vert \cdots \right\vert $ and $\left\vert \cdots \right\vert ^{\prime
} $ induce the same twisting system $R$ if and only if their restriction to $%
T$ are equal.
\end{theorem}

\begin{fact}[The twisted tensor product of thin categories.]
Let $R$ be a twisting system between two thin categories $\boldsymbol{B}$
and $\boldsymbol{A.}$ By the preceding theorem, $R$ is simple and there are $%
T$ and $|\boldsymbol{\cdots }|:S^{3}$ $\rightarrow S$ such that the
conditions (i)-(v) hold. In particular, the twisted tensor product of these
categories exists. By definition, we have $_{x}(\boldsymbol{A}\otimes _{R}%
\boldsymbol{B)}_{y}=$ $\coprod\nolimits_{u\in S}(_{x}A_{u}\times
{}_{u}B_{z}).$ We can identify this set with
\begin{equation*}
_{x}S_{y}:=\{u\in S\mid {}_{x}A_{u}B_{y}\neq \emptyset \}.
\end{equation*}%
For $u\in {}_{x}S_{y}$ and $v\in {}_{y}S_{z}$ we have $(u,y,v)\in T.$ Thus $%
_{u}R_{v}^{y}(_{u}f_{y},{}_{y}g_{v})=(_{u}g_{|uyv|},{}_{|uyv|}f_{v}),$ so
the composition in $\boldsymbol{A}\otimes _{R}\boldsymbol{B}$ is given by%
\begin{equation*}
(_{x}g_{u},{}_{u}f_{y})\circ \left( _{y}g_{v},{}_{v}f_{z}\right)
=(_{x}g_{u}\circ {}_{u}g_{\left\vert uyv\right\vert },{}_{\left\vert
uyv\right\vert }f_{v}\circ {}_{v}f_{z})=(_{x}g_{|uyv|},{}_{|uyv|}f_{z}).
\end{equation*}%
Let $\boldsymbol{C}(S,T,\left\vert \cdots \right\vert )$ be the category
whose objects are the elements of $S.$ By definition, the hom-set $_{x}%
\boldsymbol{C}(S,T,\left\vert \cdots \right\vert )_{y}$ is $_{x}S_{y},$ the
identity map of $x\in S$ is $x$ itself and the composition is given by%
\begin{equation*}
\circ :{}_{x}S_{y}\times {}_{y}S_{z}\rightarrow {}_{x}S_{z},\quad u\circ
v=|uyv|.
\end{equation*}%
Therefore, we have just proved that $\boldsymbol{A}\otimes _{R}\boldsymbol{B}
$ and $\boldsymbol{C}(S,T,|\cdots|)$ are isomorphic.
\end{fact}

\begin{remark}
Let $\boldsymbol{C}$ be a small category. Let $S$ denote the set of objects
in $\boldsymbol{C}.$ The category $C$ factorizes through two thin categories
if and only if there are $T\subseteq S$ and $\left\vert \cdots \right\vert
:S^{3}\rightarrow S$ as in the previous subsection such that $\boldsymbol{C}$
is isomorphic to $\boldsymbol{C}(S,T,\left\vert \cdots \right\vert ).$
\end{remark}

\begin{fact}[Twisting systems between posets.]
Any poset is a thin category, so we can apply Theorem \ref{thm:thin} to
characterize a twisting system $R$ between two posets $\boldsymbol{B}%
:=(S,\preceq )$ and $\boldsymbol{A}:=(S,\leq )$. In this setting the
corresponding set $T$ contains all $(x,y,z)\in S^{3}$ such that $x\preceq y$
and $y\leq z.$ For simplicity, we shall write this condition as $x\preceq
y\leq z.$ A similar notation will be used for arbitrarily long sequences of
elements in $S.$ For instance, $x\leq y\preceq z\preceq t\leq u$ means that $%
x\leq y,$ $y\preceq z,$ $z\preceq t$ and $t\leq u$. The function $\left\vert
{\cdots }\right\vert $ must satisfies the following conditions:

\begin{enumerate}
\item[(i)] If $x\preceq y\leq z$ then $x\leq |xyz|\preceq z.$

\item[(ii)] If $x\preceq y\preceq z\leq t$ then $|xy|yzt||=|xzt|$.

\item[(iii)] If $x\preceq y\leq z\leq t$ then $||xyz|zt|=|xyt|.$

\item[(iv)] If $x\leq y$ then $|xxy|=y$.

\item[(v)] If $x\preceq y$ then $|xyy|=x.$
\end{enumerate}

In the case when the posets $\leq $ and $\preceq $ are identical, an example
of function $|{\cdots }|:S^{3}$ $\rightarrow S$ that satisfies the above
conditions is given by $|xyz|=z,$ if $y\neq z,$ and $|xyz|=x,$ otherwise.
\end{fact}

\begin{fact}[Example of twisting map between two groupoids.]
Let $\boldsymbol{A}$ be a groupoid with two objects, $S=\{1,2\}.$ The
hom-sets of $\boldsymbol{A}$ are the following:
\begin{equation*}
_{1}A_{2}=\{u\},\quad _{2}A_{1}=\{u^{-1}\},\quad _{1}A_{1}=\{Id_{1}\},\quad
_{2}A_{2}\ =\{Id_{2}\}.
\end{equation*}%
Note that $\boldsymbol{A}$ is thin. We set $\boldsymbol{B}:=\boldsymbol{A}$
and we take $R$ to be a twisting system between $\boldsymbol{B}$ and $%
\boldsymbol{A}.$ By Theorem \ref{thm:thin} there are $T$ and $|{\cdots }%
|:S^{3}\rightarrow S$ that satisfies the conditions (i)-(v) in \S \ref%
{fa:thin}. Since all sets $_{x}B_{y}A_{z}={}_{x}B_{y}\times {}_{y}A_{z}$ are
nonempty it follows that $T=S.$ Thus $|xxy|=y$ and $|xyy|=x$, for all $%
x,y\in S$. There are two triples $(x,y,z)\in S^{3}$ such that $x\neq y$ and $%
y\neq z$, namely $(1,2,1)$ and $(2,1,2).$ Hence we have to compute $|121|$
and $|212|.$ If we assume that $|121|=1$, then%
\begin{equation*}
1=|221|=|21|121||=|211|=2,
\end{equation*}%
so we get a contradiction. Thus $|121|=2,$ and proceeding in a similar way
one proves that $|212|=1.$ It is easy to check that $|{\cdots }|$ satisfies
the required conditions, so there is only one twisting map ${R}$ between $%
\boldsymbol{A}$ and itself. Since $\boldsymbol{A}$ is a groupoid, the
corresponding bicrossed product $\boldsymbol{C}:=\boldsymbol{A}\Join
\boldsymbol{A}$ is a groupoid as well, see the subsection (\ref{fa:gr}). By
definition,
\begin{equation*}
_{1}C_{1}=\textstyle\coprod\limits_{x\in \{1,2\}}{}_{1}A_{x}\times
{}_{x}A_{1}=\{(Id_{1},Id_{1}),(u,u^{-1})\}.
\end{equation*}%
Analogously one shows that%
\begin{equation*}
_{1}C_{2}=\{(Id_{1},u),(u,Id_{2})\},\quad
_{2}C_{1}=\{(Id_{2},u^{-1}),(u^{-1},Id_{1})\}\quad \text{and\quad }%
_{2}C_{2}=\{(Id_{2},Id_{2}),(u^{-1},u)\}.
\end{equation*}%
By construction of the twisting map ${R}$ we get $_{1}R_{1}^{2}(u,u^{-1})\in
{}_{1}A_{|121|}\times {}_{|121|}A_{1}=\{(u,u^{-1})\}$. The other maps $%
_{x}R_{z}^{y}$ can be determined analogously. The complete structure of this
groupoid is given in the picture below, where we used the notation $%
f:=(u,u^{-1})$ and $g:=(Id_{1},u)$.
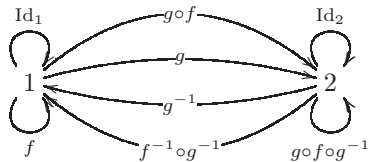
\begin{equation*}
\xy (-20,0)*+{1}="a", (20,0)*+{2}="b", \ar@(ul,ur) "a";"a"^{\text{Id}_1} \ar%
@(dl,dr) "a";"a"_{f} \ar@(ul,ur) "b";"b"^{\text{Id}_2} \ar@(dl,dr)
"b";"b"_{g\circ f\circ g^{-1}} \ar@(ul,ur) "b";"b" \ar@(dl,dr) "b";"b" \ar%
@/^2ex/ "a";"b"|{g}\ar@/^6ex/ "a";"b"|{g\circ f} \ar@/^2ex/ "b";"a"|{g^{-1}}%
\ar@/^6ex/ "b";"a"|{f^{-1}\circ g^{-1}} \endxy
\end{equation*}
Note that $f^{2}=Id_{1}$ and $g^{-1}=(Id_{2},u^{-1})$. Now we can say easily
which arrow corresponds to a given morphism in $\boldsymbol{C}$, as in each
home-set we have identified at least one element.
\end{fact}

\noindent\textbf{Acknowledgments.} The first named author was financially supported
by the funds of the Contract POSDRU/6/1.5/S/12. The second named author was financially supported
by CNCSIS, Contract 560/2009 (CNCSIS code ID\_69).

\end{document}